\documentclass[a4paper,11pt, reqno]{amsart}
\usepackage[utf8]{inputenc}
\usepackage{amsmath}
\usepackage{amssymb}
\usepackage{amsthm}
\usepackage{amsrefs}
\usepackage{a4wide}
\usepackage{paralist}
\usepackage{marvosym}
\usepackage{esint}
\usepackage{dsfont}
\usepackage{extarrows}
\usepackage{bm} 
\usepackage{enumitem}
\usepackage{calc}

\newtheorem{theorem}{Theorem}[section]
\numberwithin{theorem}{section}
\newtheorem{lemma}[theorem]{Lemma}
\newtheorem{proposition}[theorem]{Proposition}
\newtheorem{corollary}[theorem]{Corollary}

\newtheorem{remark}[theorem]{Remark}

\DeclareMathOperator{\E}{\mathds{E}}
\DeclareMathOperator{\Var}{Var}
\DeclareMathOperator{\N}{\mathbb{N}}

\DeclareMathOperator{\R}{\mathbb{R}}

\DeclareMathOperator{\1}{\mathds{1}}
\DeclareMathOperator{\supp}{supp}
\DeclareMathOperator{\dom}{dom}
\DeclareMathOperator{\Po}{Po}

\newcommand{\ellnorm}[3]{\Vert{#3}\Vert_{\ell^{#1}(\N)^{\otimes{#2}}}}

\newcommand{\absolute}[1]{\vert{#1}\vert}
\newcommand{\Bigabsolute}[1]{\Big\vert{#1}\Big\vert}
\newcommand{\norm}[1]{\|{#1}\|}

\date{\today}
\numberwithin{equation}{section}

\usepackage{hyperref}
\allowdisplaybreaks

\title[Poisson approximation of Rademacher functionals]{Poisson approximation of Rademacher functionals by the Chen-Stein method and Malliavin calculus}

\author[K. Krokowski]{Kai Krokowski}
\address{Kai Krokowski: Faculty of Mathematics, NA 3/28, Ruhr University Bochum, Germany}
\email{kai.krokowski@rub.de}

\date{}

\begin{document}

\begin{abstract}
New bounds on the total variation distance between the law of integer valued functionals of possibly non-symmetric and non-homogeneous infinite Rademacher sequences and the Poisson distribution are established. They are based on a combination of the Chen-Stein method and a discrete version of Malliavin calculus. We give some applications to shifted discrete multiple stochastic integrals.
\end{abstract}

\maketitle

\section{Introduction}
Stein's method and the Malliavin calculus have been combined for the first time by Nourdin and Peccati in the initial paper \cite{NouPec} in order to derive explicit bounds on the error in the normal and Gamma approximation of functionals of general Gaussian processes. This new approach to Stein's method, also known as the Malliavin-Stein method, has also been used to deduce quantitative central limit theorems for functionals of general Poisson measures (see \cite{PecSolTaqUtz}) and for functionals of infinite symmetric Rademacher sequences (see \cite{NouPecRei} and \cite{KroReiThae1}). Here, the term symmetric Rademacher sequence refers to a sequence of independent and identically distributed random variables taking the values $+1$ and $-1$ with probability $1/2$ each.\medskip

The results in \cite{NouPecRei} and \cite{KroReiThae1} are based on a product formula for multiple stochastic integrals (see Proposition 2.9 in \cite{NouPecRei}), whose proof relies on the simplicity of the underlying symmetric Rademacher sequence. The findings of \cite{KroReiThae1} were further developed in \cite{KroReiThae2}, where a second order Poincar\'{e} type bound on the Kolmogorov distance between the law of functionals of possibly non-symmetric and non-homogeneous infinite Rademacher sequences and the standard normal distribution was derived. For analogues of such second order Poincar\'{e} type inequalities in the Gaussian and Poisson case see \cite{NouPecReiSecOrdPoi} and \cite{LasPecSchu}, respectively. One advantage of the bound in \cite{KroReiThae2} is that it can be further evaluated without the need of a product formula for multiple stochastic integrals.\medskip

Poisson approximation by a combination of the Chen-Stein method and Malliavin calculus has first been tackled in \cite{PecPoi}, where the author computed explicit bounds on the total variation distance between the law of integer valued functionals of general Poisson measures and a Poisson distribution. Furthermore, sufficient conditions for the convergence in distribution of suitably shifted multiple stochastic integrals to a Poisson random variable and rates of convergence for the Poisson approximation of statistics associated with geometric random graphs were covered. For further works in the framework of the Chen-Stein method and Malliavin calculus see, e.g., \cite{BouPec}, \cite{Tor1} and \cite{Tor2}.\medskip

The purpose of this paper is to combine the Chen-Stein method and a discrete version of Malliavin calculus (as developed in \cite{Pri2}), and thus, to continue the work of \cite{KroReiThae1} and \cite{KroReiThae2} to the case of Poisson approximation. A general bound on the total variation distance between the law of integer valued functionals of possibly non-symmetric and non-homogeneous infinite Rademacher sequences and the Poisson distribution is shown (see Theorem \ref{Main theorem}). Applications to shifted multiple stochastic integrals are considered (see Theorem \ref{J_1 theorem} and Corollary \ref{Bernoulli sums corollary} as well as Theorem \ref{J_m theorem} and Corollary \ref{J_2 corollary}).
For this, a generalization of the product formula from \cite{NouPecRei} to multiple stochastic integrals based on an underlying possibly non-symmetric and non-homogeneous Rademacher sequence is proved (see Proposition \ref{Multiplication formula proposition}). In addition, using the techniques provided in \cite{KroReiThae2}, a second order Poincar\'{e} type inequality is deduced from the general bound (see Theorem \ref{Second order Poincare inequality}).\medskip

The remainder of this paper is built up as follows. Section 2 collects the bases of the discrete Malliavin calculus as well as the product formula for multiple stochastic integrals. Furthermore, a short introduction to the Chen-Stein method is given. Section 3 contains all of the main results and their proofs. Section 4 serves as appendix and bears the proof of the product formula and, additionally, a standard approximation argument that is used within some of the proofs in this paper.\medskip

The authors of \cite{PriTor} have also developed bounds on the total variation distance between the law of integer valued functionals of possibly non-symmetric and non-homogeneous infinite Rademacher sequences and the Poisson distribution by using a generalization of the product formula for multiple stochastic integrals in \cite{NouPecRei} as well. In particular, Theorem \ref{Main theorem} and Corollary \ref{Main corollary} here are related to Theorem 6.3 in \cite{PriTor}, Theorem \ref{J_1 theorem} and Corollary \ref{Bernoulli sums corollary} are related to Theorem 7.1 in \cite{PriTor}, and Theorem \ref{J_m theorem} and Remark \ref{J_2 remark} are related to Theorem 8.2 and Proposition 8.3, respectively, in \cite{PriTor}. However, the corresponding results of \cite{PriTor} and this paper were worked out independently of each other and differ (see, e.g., Remark \ref{PriTor remark}). In addition, we contribute a second order Poincar\'{e} type bound which is not provided in \cite{PriTor}.

\section{Preliminaries}

\subsection{Rademacher sequences}
Let $p:=(p_k)_{k \in \N}$ be a sequence of success probabilities fulfilling $0 < p_k <1$, for every $k \in \N$, and let $q:=(q_k)_{k \in \N}$ be the corresponding sequence of failure probabilities with $q_k := 1-p_k$, for every $k \in \N$. Furthermore, let $(\Omega, \mathcal{F}, P)$ be a probability space with $\Omega:=\{-1,+1\}^{\N}$, $\mathcal{F}:=\mathcal{P}(\{-1,+1\})^{\otimes \N}$ and $P:=\bigotimes_{k=1}^\infty (p_k \delta_{+1} + q_k \delta_{-1})$. Now, let $X:=(X_k)_{k\in\N}$ be a sequence of independent random variables defined on $(\Omega, \mathcal{F}, P)$ by $X_k(\omega):=\omega_k$, for every $k \in \N$ and $\omega := (\omega_k)_{k\in\N} \in \Omega$. Here, we refer to the sequence $X$ as (possibly non-symmetric and non-homogeneous) Rademacher sequence.
In the following, we will introduce discrete multiple stochastic integrals on the basis of our Rademacher sequence. To this end, we also define the standardized sequence $Y := (Y_k)_{k\in\N}$ with
\begin{align*}
Y_k := (X_k-\E[X_k])/\sqrt{\Var(X_k)} = (X_k-p_k+q_k)/(2\sqrt{p_kq_k}),
\end{align*}
for every $k \in \N$.

\subsection{Kernels and contractions}
Let $\kappa$ be the counting measure on $\N$. We put $\ell^2(\N)^{\otimes n} := L^2(\N^n,\mathcal{P}(\N)^{\otimes n}, \kappa^{\otimes n})$, for every $n \in \N$. In the following, we refer to the elements of $\ell^2(\N)^{\otimes n}$ as kernels. Let $\ell^2(\N)^{\circ n}$ denote the subset of $\ell^2(\N)^{\otimes n}$ of symmetric kernels. Furthermore, let $\ell_0^2(\N)^{\otimes n}$ denote the subset of kernels vanishing on diagonals, i.e.\ vanishing on the complement of the set $\Delta_n:=\{ (i_1, \dotsc, i_n) \in \N^n : \text{$i_j \neq i_k$ for $j \neq k$} \}$. We then put $\ell_0^2(\N)^{\circ n} := \ell^2(\N)^{\circ n} \cap \ell_0^2(\N)^{\otimes n}$. For $n,m \in \N$, take two kernels $f \in \ell_0^2(\N)^{\circ n}$ and $g \in \ell_0^2(\N)^{\circ m}$. Now, for $r=0, \dotsc, n \wedge m$ and $\ell = 0, \dotsc, r$, the contraction of $f$ and $g$ is defined by
\begin{align*}
&f \star_r^\ell g (i_1, \dotsc, i_{n-r}, k_1, \dotsc, k_{r-\ell}, j_1, \dotsc, j_{m-r})\\
&:=\sum_{(a_1, \dotsc, a_\ell) \in \Delta_\ell} f(i_1, \dotsc, i_{n-r}, k_1, \dotsc, k_{r-\ell}, a_1, \dotsc, a_\ell)g(j_1, \dotsc, j_{m-r}, k_1, \dotsc, k_{r-\ell}, a_1, \dotsc, a_\ell)\notag,
\end{align*}
that is, by identifying $r$ of the $n$ variables of $f$ with $r$ of the $m$ variables of $g$ and then integrating out $\ell$ of the $r$ identified variables with respect to the counting measure $\kappa$. Note that $f \star_r^\ell g \in \ell^2(\N)^{\otimes n+m-r-\ell}$, since $\ellnorm{2}{n+m-r-\ell}{f \star_r^\ell g} \leq \ellnorm{2}{n}{f}\ellnorm{2}{m}{g}$ (cf.\ Lemma 2.4 in \cite{NouPecRei}). Even though $f \in \ell_0^2(\N)^{\circ n}$ and $g \in \ell_0^2(\N)^{\circ m}$, the contraction $f \star_r^\ell g$ must neither be symmetric nor be vanishing on diagonals. Therefore, we define the canonical symmetrization of a function $f$ on $\N^n$ by $\widetilde{f}(i_1, \dotsc, i_n) := \frac{1}{n!} \sum_{\sigma} f(i_{\sigma(1)}, \dotsc, i_{\sigma(n)})$, where the sum runs over all permutations $\sigma$ of the set $\{ 1, \dotsc, n\}$. Note that, if $f \in \ell^2(\N)^{\otimes n}$, then $\widetilde{f} \in \ell^2(\N)^{\otimes n}$, since $\ellnorm{2}{n}{\widetilde{f}} \leq \ellnorm{2}{n}{f}$.

\subsection{Discrete multiple stochastic integrals and chaos representation}
For $n \in \N$ and $f \in \ell_0^2(\N)^{\circ n}$, we define the discrete multiple stochastic integral of order $n$ of $f$ by
\begin{align}\label{Stochastic integral}
J_n(f) &:= \sum_{(i_1, \dotsc, i_n)\in\N^n} f(i_1, \dotsc, i_n)Y_{i_1} \cdot \dotsc \cdot Y_{i_n} \notag\\
&\phantom{:}= \sum_{(i_1, \dotsc, i_n)\in\Delta_n} f(i_1, \dotsc, i_n)Y_{i_1} \cdot \dotsc \cdot Y_{i_n} \notag\\
&\phantom{:}= n!\sum_{1 \leq i_1 < \dotsc < i_n < \infty} f(i_1, \dotsc, i_n)Y_{i_1} \cdot \dotsc \cdot Y_{i_n}.
\end{align}
In addition, we put $\ell^2(\N)^{\otimes 0} := \R$ and $J_0(c):=c$, for every $c \in \R$.\medskip

For every $n \in \N$, the subspace $\{ J_n(f) : f \in \ell_0^2(\N)^{\circ n} \}$ of $L^2(\Omega)$ is called the Rademacher chaos of order $n$. Now, every square-integrable Rademacher functional $F \in L^2(\Omega)$ admits a unique decomposition of the form
\begin{align}\label{Chaos representation}
F = \E[F] + \sum_{n=1}^\infty J_n(f_n)
\end{align}
with $f_n \in \ell_0^2(\N)^{\circ n}$, for every $n \in \N$ (cf.\ Proposition 6.7 in \cite{Pri2}). We call \eqref{Chaos representation} the chaos representation of $F$, where the series converges in $L^2(\Omega)$.\medskip

We will now prepare for the presentation of a product formula for discrete stochastic integrals. The following observation is crucial to derive such a product formula (cf.\ Chapter 5 in \cite{Pri2}).
\begin{lemma}\label{Structure equation lemma}
For every $k \in \N$, $Y_k^2$ admits the chaos representation
\begin{align}\label{Structure equation}
Y_k^2 = 1 + \varphi_k Y_k,
\end{align}
where the sequence $\varphi:=(\varphi_k)_{k \in \N}$ is defined by $\varphi_k:=(q_k-p_k)/\sqrt{p_kq_k}$, for every $k \in \N$.
\end{lemma}

For $n,m \in \N$, take two kernels $f \in \ell_0^2(\N)^{\circ n}$ and $g \in \ell_0^2(\N)^{\circ m}$. Now, for $r=1, \dotsc, n \wedge m$ and $\ell=0, \dotsc, r-1$, we define the weighted contraction of $f$ and $g$ by
\begin{align*}
&\varphi^{*r-\ell}(f \star_r^\ell g) (i_1, \dotsc, i_{n-r}, k_1, \dotsc, k_{r-\ell}, j_1, \dotsc, j_{m-r})\\
&:=\varphi_{k_1} \cdot \dotsc \cdot \varphi_{k_{r-\ell}} f \star_r^\ell g(i_1, \dotsc, i_{n-r}, k_1, \dotsc, k_{r-\ell}, j_1, \dotsc, j_{m-r}).
\end{align*}
Note that the indices $k_1, \dotsc, k_{r-\ell}$ of the factors in the product $\varphi_{k_1} \cdot \dotsc \cdot \varphi_{k_{r-\ell}}$ are the $r-\ell$ variables of the contraction $f \star_r^\ell g$ that are identified but not integrated out. For $r = 0, \dotsc, {n \wedge m}$, we further define
\begin{align*}
&\varphi^{*0}(f \star_r^r g) (i_1, \dotsc, i_{n-r}, j_1, \dotsc, j_{m-r}) := f \star_r^r g(i_1, \dotsc, i_{n-r}, j_1, \dotsc, j_{m-r}).
\end{align*}

Now, the following proposition states a formula for the product of discrete multiple stochastic integrals. Note that this is a generalization of Proposition 2.9 in \cite{NouPecRei} to the case of stochastic integrals based on a possibly non-symmetric and non-homogeneous infinite Rademacher sequence. We refer to the appendix for a proof of the statement. Also note that the following Proposition \ref{Multiplication formula proposition} corresponds to Proposition 5.1 in \cite{PriTor}.
\begin{proposition}[Product formula]\label{Multiplication formula proposition}
Let $n,m\in \N$ and $f\in\ell_0^2(\N)^{\circ n}, g\in\ell_0^2(\N)^{\circ m}$. Furthermore, let $\widetilde{(\varphi^{*r-\ell}(f\star_r^\ell g))}\1_{\Delta_{n+m-r-\ell}} \in \ell_0^2(\N)^{\circ n+m-r-\ell}$, for every $r=1, \dotsc, n \wedge m$ and $\ell=0, \dotsc, r-1$. Then,
\begin{align}
J_n(f)J_m(g) &= \sum_{r=0}^{n \wedge m} r! \binom{n}{r} \binom{m}{r} \sum_{\ell=0}^r \binom{r}{\ell} J_{n+m-r-\ell} \Big( \widetilde{(\varphi^{*r-\ell}(f\star_r^\ell g))}\1_{\Delta_{n+m-r-\ell}} \Big) \label{Multiplication formula equation 1}\\[5pt]
&= \sum_{r=0}^{n \wedge m} r! \binom{n}{r} \binom{m}{r} J_{n+m-2r}\left( \widetilde{ \left( f \star_r^r g \right)} \1_{\Delta_{n+m-2r}} \right) \notag\\
&\phantom{{}={}}+\sum_{r=1}^{n \wedge m} r! \binom{n}{r} \binom{m}{r} \sum_{\ell=0}^{r-1} \binom{r}{\ell} J_{n+m-r-\ell} \Big( \widetilde{(\varphi^{*r-\ell}(f\star_r^\ell g))}\1_{\Delta_{n+m-r-\ell}} \Big) \label{Multiplication formula equation 2},
\end{align}
where we put $\1_{\Delta_0}:=1$.
\end{proposition}

\begin{remark}\leavevmode\rm
\begin{itemize}[align=left,labelwidth=\widthof{(iI)},leftmargin=\labelwidth+\labelsep]
\item[(i)] In Proposition \ref{Multiplication formula proposition}, sufficient conditions for $\widetilde{(\varphi^{*r-\ell}(f\star_r^\ell g))}\1_{\Delta_{n+m-r-\ell}}$ to be an element of $\ell^2(\N)^{\otimes n+m-r-\ell}$,  for every $r=1, \dotsc, n \wedge m$ and $\ell=0, \dotsc, r-1$, are given, e.g., if the sequence $\varphi$ is either constant or fulfills $\norm{\varphi}_{\ell^2(\N)} < \infty$.

\item[(ii)] While we will use \eqref{Multiplication formula equation 1} in further applications, the representation of the product formula in \eqref{Multiplication formula equation 2} exhibits the relation between the general case of a possibly non-symmetric and non-homogeneous Rademacher sequence and the case of a symmetric Rademacher sequence. In the case of a symmetric Rademacher sequence $X$, i.e.\ $p_k = q_k = 1/2$, for every $k \in \N$, the coefficients $\varphi_k$ of the chaos representation of $Y_k^2$ in \eqref{Structure equation} vanish, for every $k \in \N$, so that Proposition \ref{Multiplication formula proposition} reproduces Proposition 2.9 in \cite{NouPecRei}.
\end{itemize}
\end{remark}

The next corollary states an isometry formula for stochastic integrals as seen in Proposition 4.2 in \cite{Pri2}. Note that this is also an immediate conclusion from the product formula in Proposition \ref{Multiplication formula proposition}, since, for every $n \in \N$ and $f \in \ell_0^2(\N)^{\circ n}$, $\E[J_n(f)]=0$.
\begin{corollary}[Isometry formula]\label{Isometry formula corollary}
Let $n,m\in \N$ and $f\in\ell_0^2(\N)^{\circ n}, g\in\ell_0^2(\N)^{\circ m}$. Then,
\begin{align}\label{Isometry formula}
\E[J_n(f)J_m(g)]=
\begin{cases}
n! \langle f,g \rangle_{\ell^2(\N)^{\otimes n}}, &\text{if \,} n=m,\\
0, &\text{if \,} n \neq m.
\end{cases}
\end{align}
\end{corollary}

\subsection{Discrete Malliavin calculus}
We start by defining the discrete gradient operator $D$. For every $\omega = (\omega_1, \omega_2, \dotsc) \in \Omega$ and $k \in \N$, let $\omega_+^k:=(\omega_1, \dotsc, \omega_{k-1}, +1, \omega_{k+1}, \dotsc)$ and $\omega_-^k:=(\omega_1, \dotsc, \omega_{k-1}, -1, \omega_{k+1}, \dotsc)$.
Furthermore, for every $F \in L^1(\Omega)$, $\omega \in \Omega$ and $k \in \N$, let $F_k^+(\omega) := F(\omega_+^k)$ and $F_k^-(\omega) := F(\omega_-^k)$. For $F \in L^1(\Omega)$, the discrete gradient operator is defined by $DF:=(D_kF)_{k \in \N}$ with
\begin{align}\label{Pathwise gradient}
D_kF:=\sqrt{p_kq_k}(F_k^+-F_k^-),
\end{align}
for every $k \in \N$. Note that it immediately follows from \eqref{Pathwise gradient} that, for every $k \in \N$, $D_kF$ is independent of $X_k$. Now, let $F \in L^2(\Omega)$ have the chaos representation $F = \E[F] + \sum_{n=1}^\infty J_n(f_n)$ with kernels $f_n \in \ell_0^2(\N)^{\circ n}$, for every $n \in \N$. Then, for every $k \in \N$, $D_kF \in L^2(\Omega)$ and has the chaos representation 
\begin{align*}
D_kF=\sum_{n=1}^\infty nJ_{n-1}(f_n(\, \cdot \, ,k)),
\end{align*}
where, for every $n \in \N$, $f_n(\, \cdot \, ,k) \in \ell_0^2(\N)^{\circ n-1}$ denotes the kernel $f_n$ with one of its components fixed, thus as a function in only $n-1$ variables (cf.\ Chapter 2.3 in \cite{KroReiThae2}). In addition, for $F \in L^1(\Omega)$ and $m \in \N$, the iterated discrete gradient operator of order $m$ is defined by $D^mF := (D_{k_1, \dotsc, k_m}^mF)_{k_1, \dotsc, k_m \in \N}$ with $D_{k_1, \dotsc, k_m}^mF := D_{k_m}(D_{k_1, \dotsc, k_{m-1}}^{m-1}F)$, for every $k_1, \dotsc, k_m \in \N$, where we put $D_{k_1, \dotsc, k_0}^0F:=F$. Given $F \in L^2(\Omega)$ with chaos representation $F=\E[F]+\sum_{n=1}^\infty J_n(f_n)$ as above and $m \in \N$, we say $F \in \dom(D^m)$, if
\begin{align}\label{F in dom(D)}
\E [ \ellnorm{2}{m}{D^mF}^2 ] = \sum_{n=m}^\infty \frac{n!}{(n-m)!} n!  \ellnorm{2}{n}{f_n}^2 < \infty.
\end{align}

We will now define the discrete divergence operator $\delta$ and its domain $\dom(\delta)$. For $n \in \N$ and $f_{n} \in \ell_0^2(\N)^{\circ n-1} \otimes \ell^2(\N)$ we consider the sequence $u:=(u_k)_{k \in \N}$ with $u_k:= \sum_{n=1}^\infty J_{n-1}(f_n(\, \cdot \, ,k))$, for every $k \in \N$. For such a sequence $u$, we say $u \in \dom(\delta)$, if
\begin{align}\label{u in dom(delta)}
\sum_{n=1}^\infty n! \ellnorm{2}{n}{\widetilde{f}_n \1_{\Delta_n}}^2 < \infty.
\end{align}
For $u \in \dom(\delta)$, the discrete divergence operator $\delta$ is then defined by
\begin{align*}
\delta(u):= \sum_{n=1}^\infty J_{n}(\widetilde{f}_n \1_{\Delta_n}).
\end{align*}
Note that \eqref{u in dom(delta)} is equivalent to $\E[(\delta(u))^2] < \infty$. Now, $\delta$ is the adjoint of $D$ (cf.\ Proposition 9.2 in \cite{Pri2}). 

\begin{lemma}
Let $F \in \dom(D)$ and $u \in \dom(\delta)$. Then,
\begin{align}\label{Adjointness}
\E[F \delta(u)] = \E[\langle DF, u \rangle_{\ell^2(\N)}].
\end{align}
\end{lemma}

Next, we define the discrete Ornstein-Uhlenbeck operator $L$ and its (pseudo-)inverse $L^{-1}$. Given $F \in L^2(\Omega)$, again with chaos representation $F=\E[F]+\sum_{n=1}^\infty J_n(f_n)$ as above, we say $F \in \dom(L)$, if
\begin{align*}
\sum_{n=1}^\infty n^2 n!  \ellnorm{2}{n}{f_n}^2 < \infty.
\end{align*}
For $F \in \dom(L)$, the discrete Ornstein-Uhlenbeck operator $L$ is then defined by
\begin{align*}
LF := -\sum_{n=1}^\infty n J_n(f_n).
\end{align*}
For centered $F \in L^2(\Omega)$, its (pseudo-)inverse is defined by 
\begin{align*}
L^{-1}F := -\sum_{n=1}^\infty \frac{1}{n} J_n(f_n).
\end{align*}

The following lemma states the relation between the operators $D$, $\delta$ and $L$ (cf.\ Chapter 10 in \cite{Pri2}).
\begin{lemma}
It holds that
\begin{align}\label{L = -delta D}
L = -\delta D.
\end{align}
\end{lemma}

Finally, we present an integration by parts formula, which is one of the main contributions to the discrete Malliavin-Stein method.
\begin{lemma}[Integration by parts formula]
Let $F,G \in \dom(D)$. Then,
\begin{align}\label{Integration by parts formula}
\E[(F-\E[F])G] = \E[\langle -DL^{-1}(F - \E[F]), DG \rangle_{\ell^2(\N)}].
\end{align}
\end{lemma}
\begin{proof} Relation \eqref{L = -delta D} and the adjointness of $D$ and $\delta$ in \eqref{Adjointness} yield 
\begin{align*}
&\E[(F-\E[F])G] = \E[LL^{-1}(F - \E[F])G]\\
&= \E[-\delta D L^{-1}(F - \E[F])G]\\
&= \E[\langle -DL^{-1}(F - \E[F]), DG \rangle_{\ell^2(\N)}].
\end{align*}
\end{proof}

\subsection{The Chen-Stein method}
Stein's method for Poisson approximation, also known as the Chen-Stein method, has been introduced by Chen in \cite{Che}. Since then, the method was further developed by Barbour and others, see, e.g., \cite{BarHolJan}. The starting point of the method is the following characterization of a Poisson distribution. A random variable $Z$ has a Poisson distribution with mean $\lambda > 0$, if and only if, for every bounded function $f:\N_0 := \N \cup \{ 0 \} \rightarrow \R$,
\begin{align*}
\E[\lambda f(Z+1) - Zf(Z)] = 0.
\end{align*}
Now, the main idea is to set the total variation distance between the law of a given random variable and a Poisson distribution in relation to the characterization above. The link to do so is given by the Chen-Stein equation. To state the equation, let $\Po(\lambda)$ be a Poisson random variable with mean $\lambda > 0$. Then, for every $A \subseteq \N_0$ and $k \in \N_0$, the Chen-Stein equation is given by
\begin{align}\label{Chen-Stein equation}
\lambda f(k+1) - kf(k) = \1_{\{k \in A\}} - P(\Po(\lambda) \in A).
\end{align}
For $k \in \N$, \eqref{Chen-Stein equation} has a unique and bounded solution $f_{\lambda, A}:\N \rightarrow \R$ with
\begin{align}\label{Chen-Stein solution}
f_{\lambda, A}(k) := \frac{(k-1)!}{\lambda^k} \sum_{j=0}^{k-1} (\1_{\{ j \in A \}} - P(\Po(\lambda) \in A))\frac{\lambda^j}{j!}.
\end{align}
Since, for $k=0$, the value of $f(0)$ does not contribute to \eqref{Chen-Stein equation}, we conventionally put $f_{\lambda, A}(0)=0$. Given a function $f:\N_0 \rightarrow \R$, we define the forward difference of $f$ by $\Delta f(k) := f(k+1)-f(k)$, for every $k \in \N_0$. Furthermore, we define the iterated forward difference of $f$ by $\Delta^2 f(k) := \Delta(\Delta f(k))$, for every $k \in \N_0$. Moreover, the supremum norm of $f$ is given by $\norm{f}_{\infty} := \sup_{k \in \N_0} \absolute{f(k)}$. The following bounds hold for the solution of the Chen-Stein equation in \eqref{Chen-Stein solution} (cf.\ Lemma 1.1.1 and Remark 1.1.2 in \cite{BarHolJan}):
\begin{align}\label{Stein factors 1}
\norm{f_{\lambda, A}}_\infty \leq 1 \wedge \sqrt{\frac{2}{e \lambda}}, \quad \norm{\Delta f_{\lambda, A}}_\infty \leq \frac{1-e^{-\lambda}}{\lambda}.
\end{align}
In addition, the relation $\norm{\Delta^2 f_{\lambda, A}}_\infty \leq 2\norm{\Delta f_{\lambda, A}}_\infty$ gives the obvious bound
\begin{align}\label{Stein factors 2}
\norm{\Delta^2 f_{\lambda, A}}_\infty \leq \frac{2(1-e^{-\lambda})}{\lambda}.
\end{align}
Note that the bound $\norm{\Delta^2 f_{\lambda, A}}_\infty \leq 2(1-e^{-\lambda})/\lambda^2$ does not follow from Theorem 1.3 in \cite{Dal} as stated in \cite{PecPoi} and \cite{PriTor}. However, Theorem 1.3 in \cite{Dal} does lead to a bound $\norm{\Delta^2 f_{\lambda, A}}_\infty \leq 2/\lambda$.

\section{Main results}

In the following, we will deduce a bound on the error in the Poisson approximation of general integer valued functionals of possibly non-symmetric and non-homogeneous infinite Rademacher sequences with respect to the total variation distance. The total variation distance between the distributions of two random variables $X$ and $Y$ with values in $\N_0$ is defined by
\begin{align*}
d_{TV}(X,Y) := \sup_{A \subseteq \N_0} \absolute{P(X \in A) - P(Y \in A)}.
\end{align*}
For a corresponding bound on the error in the Poisson approximation of integer valued functionals of general Poisson measures see Theorem 3.1 in \cite{PecPoi}. Again, note that the following Theorem \ref{Main theorem} and Corollary \ref{Main corollary} are related to Theorem 6.3 in \cite{PriTor}

\begin{theorem}\label{Main theorem}
Let $F \in \dom(D)$ with values in $\N_0$ and let $\Po(\lambda)$ be a Poisson random variable with mean $\lambda > 0$. Then,
\begin{align}\label{Main theorem equation}
&d_{TV}(F, \Po(\lambda)) \notag\\[5pt]
&\leq \Big( 1 \wedge \sqrt{\frac{2}{e\lambda}} \Big) \absolute{\lambda - \E[F]} + \frac{1-e^{-\lambda}}{\lambda} \E[\absolute{\lambda - \langle DF, -DL^{-1}(F-\E[F]) \rangle_{\ell^2(\N)}}]\notag\\
&\phantom{{}={}} +\frac{1-e^{-\lambda}}{\lambda} \E \Big[ \Big\langle \frac{1}{\sqrt{pq}} DF(DF + \sqrt{pq}X), \absolute{-DL^{-1}(F-\E[F])} \Big\rangle_{\ell^2(\N)} \Big].
\end{align}
\end{theorem}

\begin{proof}
By the Chen-Stein equation in \eqref{Chen-Stein equation} and the integration by parts formula in \eqref{Integration by parts formula}, we have, for every $A \subseteq \N_0$,
\begin{align}\label{Main theorem proof equation 1}
&P(F \in A) - P(\Po(\lambda) \in A) = \E[\lambda f_{\lambda, A}(F+1)] - \E[F f_{\lambda, A}(F)]\notag\\
&= \E[\lambda (f_{\lambda, A}(F+1) - f_{\lambda, A}(F))] - \E[(F - \E[F]) f_{\lambda, A}(F)] - \E[(\E[F] - \lambda) f_{\lambda, A}(F)]\notag\\
&= \E[\lambda \Delta f_{\lambda, A}(F)] - \E[\langle Df_{\lambda, A}(F), -DL^{-1}(F-\E[F]) \rangle_{\ell^2(\N)}] - \E[(\E[F] - \lambda) f_{\lambda, A}(F)].
\end{align}

We will now further deduce $Df_{\lambda, A}(F)$. For every $k \in \N$, we have
\begin{align}\label{Main theorem proof D_kf_h(F)}
D_kf_{\lambda, A}(F) &= \sqrt{p_kq_k}(f_{\lambda, A}(F_k^+) - f_{\lambda, A}(F_k^-))\notag\\
&= \Delta f_{\lambda, A}(F) \cdot D_kF + \sqrt{p_kq_k}(f_{\lambda, A}(F_k^+) - f_{\lambda, A}(F_k^-) - \Delta f_{\lambda, A}(F)(F_k^+ - F_k^-))\notag\\
&= \Delta f_{\lambda, A}(F) \cdot D_kF + R_k(F)
\end{align}
with
\begin{align*}
R_k(F) := \sqrt{p_kq_k}(f_{\lambda, A}(F_k^+) - f_{\lambda, A}(F_k^-) - \Delta f_{\lambda, A}(F)(F_k^+ - F_k^-)).
\end{align*}
Now, let $a,k \in \N_0$ with $k \geq a+2$. Then,
\begin{align*}
f_{\lambda, A}(k)-f_{\lambda, A}(a)-\Delta f_{\lambda, A}(a)(k-a) = \sum_{j=1}^{k-a-1} j \cdot \Delta^2 f_{\lambda, A}(k-j-1).
\end{align*}
Similarly, for every $a,k \in \N_0$ with $k \leq a-1$, one gets
\begin{align*}
f_{\lambda, A}(k)-f_{\lambda, A}(a)-\Delta f_{\lambda, A}(a)(k-a) = \sum_{j=1}^{a-k} j \cdot \Delta^2 f_{\lambda, A}(k+j-1).
\end{align*}
Moreover, for every $a \in \N_0$ and $k \in \lbrace a, a+1 \rbrace$, it holds that
\begin{align*}
f_{\lambda, A}(k)-f_{\lambda, A}(a)-\Delta f_{\lambda, A}(a)(k-a) = 0.
\end{align*} 
Thus, for every $a,k \in \N_0$,
\begin{align}\label{Main theorem proof forward difference}
\absolute{f_{\lambda, A}(k)-f_{\lambda, A}(a)-\Delta f_{\lambda, A}(a)(k-a)} \leq \frac{\norm{\Delta^2 f_{\lambda, A}}_\infty}{2} (k-a)(k-a-1).
\end{align}
We will now use \eqref{Main theorem proof forward difference} to further estimate the error term $R_k(F)$ in the chain rule at \eqref{Main theorem proof D_kf_h(F)}. Note that, for every $k \in \N$, we have
\begin{align*}
R_k(F) &= \sqrt{p_kq_k}(f_{\lambda, A}(F_k^+) - f_{\lambda, A}(F_k^-) - \Delta f_{\lambda, A}(F_k^-)(F_k^+ - F_k^-))\1_{\lbrace X_k=-1 \rbrace}\\
&\phantom{{}={}} - \sqrt{p_kq_k}(f_{\lambda, A}(F_k^-) - f_{\lambda, A}(F_k^+) - \Delta f_{\lambda, A}(F_k^+)(F_k^- - F_k^+))\1_{\lbrace X_k=+1 \rbrace}.
\end{align*}
It then follows by \eqref{Main theorem proof forward difference} that, for every $k \in \N$,
\begin{align*}
\absolute{f_{\lambda, A}(F_k^+) - f_{\lambda, A}(F_k^-) - \Delta f_{\lambda, A}(F_k^-)(F_k^+ - F_k^-)} \leq \frac{\norm{\Delta^2 f_{\lambda, A}}_\infty}{2} (F_k^+ - F_k^-)(F_k^+ - F_k^- - 1)\phantom{.}
\end{align*}
and
\begin{align*}
\absolute{f_{\lambda, A}(F_k^-) - f_{\lambda, A}(F_k^+) - \Delta f_{\lambda, A}(F_k^+)(F_k^- - F_k^+)} \leq \frac{\norm{\Delta^2 f_{\lambda, A}}_\infty}{2} (F_k^+ - F_k^-)(F_k^+ - F_k^- + 1).
\end{align*}
Thus, for every $k \in \N$,
\begin{align}\label{Main theorem proof error term}
\absolute{R_k(F)} &\leq \frac{\norm{\Delta^2 f_{\lambda, A}}_\infty}{2} \sqrt{p_kq_k} (F_k^+ - F_k^-)(F_k^+ - F_k^- - 1)\1_{\lbrace X_k=-1 \rbrace}\notag\\
&\phantom{{}\leq{}}+ \frac{\norm{\Delta^2 f_{\lambda, A}}_\infty}{2} \sqrt{p_kq_k} (F_k^+ - F_k^-)(F_k^+ - F_k^- + 1)\1_{\lbrace X_k=+1 \rbrace}\notag\\[5pt]
&= \frac{\norm{\Delta^2 f_{\lambda, A}}_\infty}{2} \sqrt{p_kq_k} (F_k^+ - F_k^-)(F_k^+ - F_k^- + X_k)\notag\\[5pt]
&= \norm{\Delta^2 f_{\lambda, A}}_\infty \frac{1}{2\sqrt{p_kq_k}} D_kF(D_kF + \sqrt{p_kq_k}X_k).
\end{align}
Putting $R(F) := (R_k(F))^{\phantom{+}}_{k \in \N}$, we then deduce from \eqref{Main theorem proof equation 1} by \eqref{Main theorem proof D_kf_h(F)} and \eqref{Main theorem proof error term} that, for every $A \subseteq \N_0$,
\begin{align*}
&\absolute{P(F \in A) - P(\Po(\lambda) \in A)}\\[5pt]
&= \vert{\E[\Delta f_{\lambda, A}(F) (\lambda - \langle DF, -DL^{-1}(F-\E[F]) \rangle_{\ell^2(\N)}]}\\
&\phantom{{}={}} {-\E[\langle R(F), -DL^{-1}(F-\E[F]) \rangle_{\ell^2(\N)})] - \E[(\E[F] - \lambda) f_{\lambda, A}(F)]}\vert\\[5pt]
&\leq \norm{\Delta f_{\lambda, A}}_\infty \E[\absolute{\lambda - \langle DF, -DL^{-1}(F-\E[F]) \rangle_{\ell^2(\N)}}]\\
&\phantom{{}={}} +\norm{\Delta^2 f_{\lambda, A}}_\infty \E \Big[ \Big\langle \frac{1}{2\sqrt{pq}} DF(DF + \sqrt{pq}X), \absolute{-DL^{-1}(F-\E[F])} \Big\rangle_{\ell^2(\N)} \Big]\\
&\phantom{{}={}} +\norm{f_{\lambda, A}}_\infty \absolute{\lambda - \E[F]}.
\end{align*}
\eqref{Main theorem equation} now follows by \eqref{Stein factors 1} and \eqref{Stein factors 2}.
\end{proof}

\begin{remark}\rm
Note that the arguments used in the proof of Theorem \ref{Main theorem} are not restricted to the choice of the total variation distance to measure the distance between the laws of $F$ and $\Po(\lambda)$. Indeed, choosing any arbitrary class $\mathcal{H}$ of bounded test functions $h: \N_0 \rightarrow \R$ would lead to a bound
\begin{align*}
&\sup_{h \in \mathcal{H}} \absolute{\E[h(F)] - \E[h(\Po(\lambda))]}\\
&\leq \norm{f_h}_\infty \absolute{\lambda - \E[F]} + \norm{\Delta f_h}_\infty \E[\absolute{\lambda - \langle DF, -DL^{-1}(F-\E[F]) \rangle_{\ell^2(\N)}}]\\
&\phantom{{}\leq{}} +\norm{\Delta^2 f_h}_\infty \E \Big[ \Big\langle \frac{1}{2\sqrt{pq}} DF(DF + \sqrt{pq}X), \absolute{-DL^{-1}(F-\E[F])} \Big\rangle_{\ell^2(\N)} \Big],
\end{align*}
where $f_h$ denotes the solution to the corresponding Chen-Stein equation. Taking, e.g., $\mathcal{H}$ as the set of all Lipschitz functions on $\N_0$ with Lipschitz constant not greater than $1$ yields the following bound on the Wasserstein distance:
\begin{align*}
d_W(F, \Po(\lambda))
&\leq \absolute{\lambda - \E[F]} + \Big( 1 \wedge \frac{8}{3\sqrt{2e\lambda}} \Big) \E[\absolute{\lambda - \langle DF, -DL^{-1}(F-\E[F]) \rangle_{\ell^2(\N)}}]\\
&\phantom{{}\leq{}} +\Big( \frac{4}{3} \wedge \frac{2}{\lambda} \Big) \E \Big[ \Big\langle \frac{1}{2\sqrt{pq}} DF(DF + \sqrt{pq}X), \absolute{-DL^{-1}(F-\E[F])} \Big\rangle_{\ell^2(\N)} \Big],
\end{align*}
where we took the bounds for $\norm{f_h}_\infty, \norm{\Delta f_h}_\infty$ and $\norm{\Delta^2 f_h}_\infty$ from Theorem 1.1 in \cite{BarXia}.
\end{remark}

The following corollary shows that we can rewrite the bound in \eqref{Main theorem equation} without resorting to the Rademacher sequence $X$. In this way, our bound here gets a representation closer to the one of the bound in Theorem 3.1 in \cite{PecPoi}.

\begin{corollary}\label{Main corollary}
Let $F \in \dom(D)$ with values in $\N_0$ and let $\Po(\lambda)$ be a Poisson random variable with mean $\lambda > 0$. Then,
\begin{align*}
&d_{TV}(F, \Po(\lambda))\\[5pt]
&\leq \Big( 1 \wedge \sqrt{\frac{2}{e\lambda}} \Big) \absolute{\lambda - \E[F]} + \frac{1-e^{-\lambda}}{\lambda} \E[\absolute{\lambda - \langle DF, -DL^{-1}(F-\E[F]) \rangle_{\ell^2(\N)}}]\\
&\phantom{{}={}} +\frac{1-e^{-\lambda}}{\lambda} \sum_{k=1}^\infty \frac{1}{\sqrt{p_kq_k}} \E[D_kF (D_kF + \sqrt{p_kq_k}(p_k-q_k)) \cdot \absolute{-D_kL^{-1}(F-\E[F])}].
\end{align*}
\end{corollary}
\begin{proof}
We only have to consider the last summand of the bound in \eqref{Main theorem equation} separately. Since, for every $k \in \N$, $D_kF(D_kF + \sqrt{p_kq_k}X_k) \geq 0$ by \eqref{Main theorem proof error term} and $D_kF$ is independent of $X_k$ for every $F \in L^1(\Omega)$, we get
\begin{align*}
&\E \Big[ \Big\langle \frac{1}{\sqrt{pq}} DF(DF + \sqrt{pq}X), \absolute{-DL^{-1}(F-\E[F])} \Big\rangle_{\ell^2(\N)} \Big]\\
&= \sum_{k=1}^\infty \frac{1}{\sqrt{p_kq_k}} \E[D_kF(D_kF + \sqrt{p_kq_k}X_k) \cdot \absolute{-D_kL^{-1}(F-\E[F])}]\\
&= \sum_{k=1}^\infty \frac{1}{\sqrt{p_kq_k}} \E[D_kF (D_kF + \sqrt{p_kq_k}(p_k-q_k)) \cdot \absolute{-D_kL^{-1}(F-\E[F])}].
\end{align*}
Plugging this into \eqref{Main theorem equation} concludes the proof.
\end{proof}

In the following, we will deduce explicit bounds on the error in the Poisson approximation of suitably shifted discrete multiple stochastic integrals of fixed order with respect to the total variation distance. We start with suitably shifted discrete multiple stochastic integrals of order 1. Again, note that the following Theorem \ref{J_1 theorem} and Corollary \ref{Bernoulli sums corollary} are related to Theorem 7.1 in \cite{PriTor}.

\begin{theorem}\label{J_1 theorem}
Let $F = \E[F] + J_1(f)$ with values in $\N_0$ and $f \in \ell^2(\N)$. Furthermore, let $\Po(\lambda)$ be a Poisson random variable with mean $\lambda > 0$. Then,
\begin{align}
&d_{TV}(F, \Po(\lambda)) \notag\\[5pt]
&\leq \Big( 1 \wedge \sqrt{\frac{2}{e\lambda}} \Big) \absolute{\lambda - \E[F]} + \frac{1-e^{-\lambda}}{\lambda} \absolute{\lambda - \norm{f}_{\ell^2(\N)}^2}\notag\\
&\phantom{{}={}} +\frac{1-e^{-\lambda}}{\lambda} \sum_{k=1}^\infty \frac{1}{\sqrt{p_kq_k}} (\absolute{f^3(k)} + \sqrt{p_kq_k}(p_k-q_k)f^2(k)) \label{J_1 theorem equation 1}\\[5pt]
&= \Big( 1 \wedge \sqrt{\frac{2}{e\lambda}} \Big) \absolute{\lambda - \E[F]} + \frac{1-e^{-\lambda}}{\lambda} \absolute{\lambda - \Var(F)}\notag\\
&\phantom{{}={}} +\frac{1-e^{-\lambda}}{\lambda} \sum_{k=1}^\infty \frac{1}{\sqrt{p_kq_k}} (f^2(k) + \sqrt{p_kq_k}(p_k-q_k)f(k)) \cdot \absolute{f(k)}. \label{J_1 theorem equation 2}
\end{align}
\end{theorem}
\begin{proof}
In order to show \eqref{J_1 theorem equation 1} and \eqref{J_1 theorem equation 2}, we have to evaluate the last two summands of the bound in \eqref{Main theorem equation}. By virtue of Corollary \ref{Main corollary}, we thus have to compute the quantities
\begin{align*}
A_1 := \E[\absolute{\lambda - \langle DF, -DL^{-1}(F-\E[F]) \rangle_{\ell^2(\N)}}]
\end{align*}
and
\begin{align*}
A_2 := \sum_{k=1}^\infty \frac{1}{\sqrt{p_kq_k}} \E[D_kF (D_kF + \sqrt{p_kq_k}(p_k-q_k)) \cdot \absolute{-D_kL^{-1}(F-\E[F])}].
\end{align*}
Now, for every $k \in \N$, we have that
\begin{align*}
D_kF = -D_kL^{-1}(F-\E[F]) = f(k).
\end{align*}
This yields
\begin{align*}
\langle DF, -DL^{-1}(F-\E[F]) \rangle_{\ell^2(\N)} = \sum_{k=1}^\infty f^2(k) = \norm{f}_{\ell^2(\N)}^2.
\end{align*}
In addition, by the isometry formula in \eqref{Isometry formula}, it follows that
\begin{align*}
\Var(F) = \norm{f}_{\ell^2(\N)}^2.
\end{align*}
Thus,
\begin{align*}
A_1 = \absolute{\lambda - \norm{f}_{\ell^2(\N)}^2} = \absolute{\lambda - \Var(F)}.
\end{align*}
Furthermore, we have
\begin{align*}
A_2 &= \sum_{k=1}^\infty \frac{1}{\sqrt{p_kq_k}} (f^2(k) + \sqrt{p_kq_k}(p_k-q_k)f(k)) \cdot \absolute{f(k)}.
\end{align*}
This concludes the proof.
\end{proof}

Now, the following corollary is a first application of Theorem \ref{J_1 theorem} and serves as an insight into the quality of our main bound in Theorem \ref{Main theorem}.

\begin{corollary}\label{Bernoulli sums corollary}
Let $(B_k)_{k \in \N}$ be a sequence of independent Bernoulli random variables with $P(B_k=1)=p_k$ and $P(B_k=0)=q_k$, for every $k \in \N$, and $\sum_{k=1}^\infty p_k < \infty$. Furthermore, let $F=\sum_{k=1}^\infty B_k$ and let $\Po(\lambda)$ be a Poisson random variable with mean $\lambda > 0$. Then,
\begin{align*}
d_{TV}(F, \Po(\lambda)) \leq \Big( 1 \wedge \sqrt{\frac{2}{e\lambda}} \Big) \Bigabsolute{\lambda - \sum_{k=1}^\infty p_k} + \frac{1-e^{-\lambda}}{\lambda} \Bigabsolute{\lambda - \sum_{k=1}^\infty p_kq_k} + \frac{2(1-e^{-\lambda})}{\lambda} \sum_{k=1}^\infty p_k^2q_k.
\end{align*}
\end{corollary}
\begin{proof}
Since, for every $k \in \N$,
\begin{align*}
B_k \overset{d}{=}\frac{X_k+1}{2},
\end{align*}
$F$ has a representation of the form $F \overset{d}{=} \E[F]+J_1(f)$ with $f \in \ell^2(\N)$. More precisely, we have
\begin{align*}
F \overset{d}{=} \sum_{k=1}^\infty \frac{X_k+1}{2} = \sum_{k=1}^\infty p_k + \sum_{k=1}^\infty \sqrt{p_kq_k}\frac{X_k+1-2p_k}{2\sqrt{p_kq_k}} = \E[F] + J_1(f)
\end{align*}
with $f(k):=\sqrt{p_kq_k}$, for every $k \in \N$. Note that $f \in \ell^2(\N)$, since
\begin{align*}
\sum_{k=1}^\infty f^2(k) = \sum_{k=1}^\infty p_kq_k \leq \sum_{k=1}^\infty p_k < \infty.
\end{align*}
According to Theorem \ref{J_1 theorem}, we thus have to evaluate the quantities
\begin{align*}
A_1:=\absolute{\lambda - \E[F]}, \quad A_2:=\absolute{\lambda - \Var(F)}
\end{align*}
and
\begin{align*}
A_3:=\sum_{k=1}^\infty \frac{1}{\sqrt{p_kq_k}} (f^2(k) + \sqrt{p_kq_k}(p_k-q_k)f(k)) \cdot \absolute{f(k)}.
\end{align*}
Now,
\begin{align*}
\E[F]=\sum_{k=1}^\infty p_k, \quad \Var(F)=\sum_{k=1}^\infty p_kq_k.
\end{align*}
Thus,
\begin{align*}
A_1 = \Bigabsolute{\lambda - \sum_{k=1}^\infty p_k}, \quad A_2 = \Bigabsolute{\lambda - \sum_{k=1}^\infty p_kq_k}.
\end{align*}
In addition,
\begin{align*}
A_3 = \sum_{k=1}^\infty p_kq_k(1+p_k-q_k) = 2\sum_{k=1}^\infty p_k^2q_k.
\end{align*}
This concludes the proof.
\end{proof}

\begin{remark}\rm\label{PriTor remark}
Note that, for $\lambda := \sum_{k=1}^\infty p_k$, the bound in Corollary \ref{Bernoulli sums corollary} yields
\begin{align*}
d_{TV}(F, \Po(\lambda)) \leq \frac{1-e^{-\lambda}}{\lambda} \sum_{k=1}^\infty p_k^2 + \frac{2(1-e^{-\lambda})}{\lambda} \sum_{k=1}^\infty p_k^2q_k \leq \frac{3(1-e^{-\lambda})}{\lambda} \sum_{k=1}^\infty p_k^2,
\end{align*}
and thus, is (up to the constant) of the quality of the classical result
\begin{align*}
d_{TV}\Big( \sum_{k=1}^n B_k, \Po(\lambda) \Big) \leq \frac{1-e^{-\lambda}}{\lambda}\sum_{k=1}^n p_k^2
\end{align*}
as discussed in Chapter 1 in \cite{BarHolJan}.
However, Corollary 7.1 in \cite{PriTor} does lead to a suboptimal result (cf.\ Chapter 7 in \cite{PriTor}).
\end{remark}

We will now turn to suitably shifted discrete stochastic integrals of order $m \geq 2$. Here, we will have to fully make use of the generalized product formula in Proposition \ref{Multiplication formula proposition}. For a corresponding result on the Poisson approximation of perturbed functionals of general Poisson measures inside a fixed chaos see Theorem 4.10 in \cite{PecPoi}. Again, note that the following Theorem \ref{J_m theorem} and Remark \ref{J_2 remark} are related to Theorem 8.2 and Proposition 8.3, respectively, in \cite{PriTor}.

\begin{theorem}\label{J_m theorem}
Let $m \geq 2$ be an integer, $F = \E[F] + J_m(f)$ with values in $\N_0$ and $f \in \ell_0^2(\N)^{\circ m}$ fulfilling $\widetilde{(\varphi^{*r-\ell}(f\star_r^\ell f))}\1_{\Delta_{n+m-r-\ell}} \in \ell_0^2(\N)^{\circ n+m-r-\ell}$, for every $r=1, \dotsc, n \wedge m$ and $\ell=0, \dotsc, r-1$. Furthermore, let $\Po(\lambda)$ be a Poisson random variable with mean $\lambda > 0$. Then,
\begin{align}
&d_{TV}(F, \Po(\lambda)) \notag\\[5pt]
&\leq \Big( 1 \wedge \sqrt{\frac{2}{e\lambda}} \Big) \absolute{\lambda - \E[F]} + \frac{1-e^{-\lambda}}{\lambda} \absolute{\lambda - \Var(F)} \notag\\
&\phantom{{}\leq{}} +\frac{1-e^{-\lambda}}{\lambda} \Big( m^2\sum_{s=1}^{2(m-1)} s! \Big\| \sum_{r=1}^m \sum_{\ell=1}^r \1_{\{ 2m-r-\ell=s \}} (r-1)! \binom{m-1}{r-1}^2 \binom{r-1}{\ell-1} \notag\\
&\phantom{\phantom{{}\leq{}} +\frac{1-e^{-\lambda}}{\lambda} \Big( m^2\sum_{s=1}^{2(m-1)} s! \Big\| \sum_{r=1}^m \sum_{\ell=1}^r} \times (\widetilde{\varphi^{*r-\ell}(f \star_{r}^\ell f)}) \1_{\Delta_{2m-r-\ell}} \Big\|_{\ell^2(\N)^{\otimes s}}^2 \Big)^{1/2} \notag\\
&\phantom{{}\leq{}} +\frac{1-e^{-\lambda}}{\lambda} \sqrt{\Var(F)} \notag\\
&\phantom{{}\leq{}+{}} \times \Big( m^3\sum_{k=1}^\infty \frac{1}{p_kq_k} ((m-1)! \ellnorm{2}{m-1}{f( \, \cdot \, ,k)}^2)^2 \notag\\
&\phantom{\phantom{{}\leq{}+{}} \times \Big(} + m^3\sum_{k=1}^\infty \frac{1}{p_kq_k} \sum_{\genfrac{}{}{0pt}{}{s=1}{s \neq m-1}}^{2(m-1)} s! \Big\| \sum_{r=1}^m \sum_{\ell=1}^r \1_{\{ 2m-r-\ell=s \}} (r-1)! \binom{m-1}{r-1}^2 \binom{r-1}{\ell-1} \notag\\ &\phantom{\phantom{\phantom{{}\leq{}+{}} \times \Big(} + m^3\sum_{k=1}^\infty \frac{1}{p_kq_k} \sum_{\genfrac{}{}{0pt}{}{s=1}{s \neq m-1}}^{2(m-1)} s! \Big\| \sum_{r=1}^m \sum_{\ell=1}^r} \times (\widetilde{\varphi^{*r-\ell}(f(\, \cdot \, ,k) \star_{r-1}^{\ell-1} f(\, \cdot \, ,k))}) \1_{\Delta_{2m-r-\ell}} \Big\|_{\ell^2(\N)^{\otimes s}}^2 \notag\\
&\phantom{\phantom{{}\leq{}+{}} \times \Big(} + m^3\sum_{k=1}^\infty \frac{1}{p_kq_k} (m-1)! \Big\| \sum_{r=1}^m \sum_{\ell=1}^r \1_{\{ 2m-r-\ell=m-1 \}} (r-1)! \binom{m-1}{r-1}^2 \binom{r-1}{\ell-1} \notag\\
&\phantom{\phantom{\phantom{{}\leq{}+{}} \times \Big(} + m^3\sum_{k=1}^\infty \frac{1}{p_kq_k} (m-1)! \Big\| \sum_{r=1}^m \sum_{\ell=1}^r} \times (\widetilde{\varphi^{*r-\ell}(f(\, \cdot \, ,k) \star_{r-1}^{\ell-1} f(\, \cdot \, ,k))}) \1_{\Delta_{2m-r-\ell}} \notag\\
&\phantom{\phantom{\phantom{{}\leq{}+{}} \times \Big(} + m^3\sum_{k=1}^\infty \frac{1}{p_kq_k} (m-1)! \Big\|} +\frac{1}{m}\sqrt{p_kq_k}(p_k-q_k) f(\, \cdot \, ,k) \Big\|_{\ell^2(\N)^{\otimes m-1}}^2 \Big)^{1/2}.
\label{J_m theorem equation}
\end{align}
\end{theorem}
\begin{proof}
It suffices to prove \eqref{J_m theorem equation} for kernels $f \in \ell_0^2(\N)^{\circ m}$ with finite support only. The general case then follows by considering the sequence of truncated kernels $(f_k)_{k \in \N}$ with $f_k:=f\1_{\{ 1, \dotsc, k\}^n}$, for every $k \in \N$, and the approximation arguments further discussed in Lemma \ref{Contraction limit} and Corollary \ref{Approximation argument} in the appendix. Again, we make use of Corollary \ref{Main corollary} and compute the quantities
\begin{align*}
A_2 := \E[\absolute{\lambda - \langle DF, -DL^{-1}(F-\E[F]) \rangle_{\ell^2(\N)}}]
\end{align*}
and
\begin{align*}
A_3 := \sum_{k=1}^\infty \frac{1}{\sqrt{p_kq_k}} \E[D_kF (D_kF + \sqrt{p_kq_k}(p_k-q_k)) \cdot \absolute{-D_kL^{-1}(F-\E[F])}].
\end{align*}
Now, for every $k \in \N$, we have
\begin{align}\label{J_m theorem proof D_kF}
D_kF = mJ_{m-1}(f(\, \cdot \, ,k)).
\end{align}
By the product formula in \eqref{Multiplication formula equation 1}, it then follows that, for every $k \in \N$,
\begin{align}\label{J_m theorem proof D_kF^2}
&(D_kF)^2 = m^2(J_{m-1}(f(\, \cdot \, ,k)))^2 \notag\\[5pt]
&= m^2\sum_{r=0}^{m-1} r! \binom{m-1}{r}^2 \sum_{\ell=0}^r \binom{r}{\ell} J_{2(m-1)-r-\ell}\Big( (\widetilde{\varphi^{*r-\ell}(f(\, \cdot \, ,k) \star_r^\ell f(\, \cdot \, ,k))}) \1_{\Delta_{2(m-1)-r-\ell}} \Big) \notag\\
&= m^2\sum_{r=1}^{m} (r-1)! \binom{m-1}{r-1}^2 \sum_{\ell=1}^r \binom{r-1}{\ell-1} J_{2m-r-\ell}\Big( (\widetilde{\varphi^{*r-\ell}(f(\, \cdot \, ,k) \star_{r-1}^{\ell-1} f(\, \cdot \, ,k))}) \1_{\Delta_{2m-r-\ell}} \Big).
\end{align}
Thus,
\begin{align}\label{J_m theorem proof sum D_kF^2}
&\sum_{k=1}^\infty (D_kF)^2 = m^2\sum_{r=1}^{m} (r-1)! \binom{m-1}{r-1}^2 \sum_{\ell=1}^r \binom{r-1}{\ell-1} J_{2m-r-\ell}\Big( (\widetilde{\varphi^{*r-\ell}(f \star_{r}^\ell f)}) \1_{\Delta_{2m-r-\ell}} \Big) \notag\\[5pt]
&= m^2\sum_{s=0}^{2(m-1)} J_s \Big( \sum_{r=1}^{m} \sum_{\ell=1}^r \1_{\{ 2m-r-\ell=s \}} (r-1)! \binom{m-1}{r-1}^2 \binom{r-1}{\ell-1} (\widetilde{\varphi^{*r-\ell}(f \star_{r}^\ell f)}) \1_{\Delta_{2m-r-\ell}} \Big) \notag\\[5pt]
&= m \cdot m! \ellnorm{2}{m}{f}^2 \notag\\
&\phantom{{}={}} +m^2\sum_{s=1}^{2(m-1)} J_s \Big( \sum_{r=1}^m \sum_{\ell=1}^r \1_{\{ 2m-r-\ell=s \}} (r-1)! \binom{m-1}{r-1}^2 \binom{r-1}{\ell-1} (\widetilde{\varphi^{*r-\ell}(f \star_{r}^\ell f)}) \1_{\Delta_{2m-r-\ell}} \Big).
\end{align}
Furthermore, for every $k \in \N$, we have
\begin{align}\label{J_m theorem proof -D_kL^-1}
-D_kL^{-1}(F-\E[F]) = J_{m-1}(f(\, \cdot \, ,k)) = \frac{1}{m} D_kF,
\end{align}
and therefore, by \eqref{J_m theorem proof sum D_kF^2}
\begin{align*}
&\langle DF, -DL^{-1}(F-\E[F]) \rangle_{\ell^2(\N)} = \frac{1}{m} \sum_{k=1}^\infty (D_kF)^2\\[5pt]
&= m! \ellnorm{2}{m}{f}^2\\
&\phantom{{}={}} +m\sum_{s=1}^{2(m-1)} J_s \Big( \sum_{r=1}^m \sum_{\ell=1}^r \1_{\{ 2m-r-\ell=s \}} (r-1)! \binom{m-1}{r-1}^2 \binom{r-1}{\ell-1} (\widetilde{\varphi^{*r-\ell}(f \star_{r}^\ell f)}) \1_{\Delta_{2m-r-\ell}} \Big).
\end{align*}
By the Cauchy-Schwarz inequality and the isometry formula in \eqref{Isometry formula}, we then get
\begin{align*}
A_2 &\leq (\E[(\lambda - \langle DF, -DL^{-1}(F-\E[F]) \rangle_{\ell^2(\N)})^2])^{1/2}\\[5pt]
&\leq \absolute{\lambda - m! \ellnorm{2}{m}{f}^2}\\
&\phantom{{}\leq{}} +\Big( m^2\sum_{s=1}^{2(m-1)} s! \Big\| \sum_{r=1}^m \sum_{\ell=1}^r \1_{\{ 2m-r-\ell=s \}} (r-1)! \binom{m-1}{r-1}^2 \binom{r-1}{\ell-1}\\
&\phantom{\phantom{{}\leq{}} +\Big( m^2\sum_{s=1}^{2(m-1)} s! \Big\| \sum_{r=1}^m \sum_{\ell=1}^r} \times (\widetilde{\varphi^{*r-\ell}(f \star_{r}^\ell f)}) \1_{\Delta_{2m-r-\ell}} \Big\|_{\ell^2(\N)^{\otimes s}}^2 \Big)^{1/2}.
\end{align*}
Using \eqref{J_m theorem proof -D_kL^-1}, the Cauchy-Schwarz inequality and \eqref{J_m theorem proof sum D_kF^2}, we further deduce
\begin{align*}
A_3 &= \frac{1}{m} \sum_{k=1}^\infty \frac{1}{\sqrt{p_kq_k}} \E[((D_kF)^2 + \sqrt{p_kq_k}(p_k-q_k)D_kF) \cdot \absolute{D_kF}]\\
&\leq \frac{1}{m} \Big( \sum_{k=1}^\infty \frac{1}{p_kq_k} \E[((D_kF)^2 + \sqrt{p_kq_k}(p_k-q_k)D_kF)^2] \Big)^{1/2} \Big( \sum_{k=1}^\infty \E[(D_kF)^2] \Big)^{1/2}\\
&= (m!\ellnorm{2}{m}{f}^2)^{1/2} \Big( \frac{1}{m} \sum_{k=1}^\infty \frac{1}{p_kq_k} \E[((D_kF)^2 + \sqrt{p_kq_k}(p_k-q_k)D_kF)^2] \Big)^{1/2}.
\end{align*}
Now, by \eqref{J_m theorem proof D_kF^2} and \eqref{J_m theorem proof D_kF}, we have
\begin{align*}
&(D_kF)^2 + \sqrt{p_kq_k}(p_k-q_k)D_kF\\[5pt]
&= m^2\sum_{r=1}^{m} (r-1)! \binom{m-1}{r-1}^2 \sum_{\ell=1}^r \binom{r-1}{\ell-1} J_{2m-r-\ell}\Big( (\widetilde{\varphi^{*r-\ell}(f(\, \cdot \, ,k) \star_{r-1}^{\ell-1} f(\, \cdot \, ,k))}) \1_{\Delta_{2m-r-\ell}} \Big)\\
&\phantom{{}={}} +m\sqrt{p_kq_k}(p_k-q_k) J_{m-1}(f(\, \cdot \, ,k))\\[5pt]
&= m^2 \sum_{\genfrac{}{}{0pt}{}{s=0}{s \neq m-1}}^{2(m-1)} J_s \Big( \sum_{r=1}^m \sum_{\ell=1}^r \1_{\{ 2m-r-\ell=s \}} (r-1)! \binom{m-1}{r-1}^2 \binom{r-1}{\ell-1}\\
&\phantom{{}= m^2 \sum_{\genfrac{}{}{0pt}{}{s=0}{s \neq m-1}}^{2(m-1)} J_s \Big( \sum_{r=1}^m \sum_{\ell=1}^r} \times (\widetilde{\varphi^{*r-\ell}(f(\, \cdot \, ,k) \star_{r-1}^{\ell-1} f(\, \cdot \, ,k))}) \1_{\Delta_{2m-r-\ell}} \Big)\\
&\phantom{{}={}} +m^2 J_{m-1} \Big( \sum_{r=1}^m \sum_{\ell=1}^r \1_{\{ 2m-r-\ell=m-1 \}} (r-1)! \binom{m-1}{r-1}^2 \binom{r-1}{\ell-1}\\
&\phantom{\phantom{{}={}} +m^2 J_{m-1} \Big( \sum_{r=1}^m \sum_{\ell=1}^r} \times (\widetilde{\varphi^{*r-\ell}(f(\, \cdot \, ,k) \star_{r-1}^{\ell-1} f(\, \cdot \, ,k))}) \1_{\Delta_{2m-r-\ell}}\\
&\phantom{\phantom{{}={}} +m^2 J_{m-1} \Big(} +\frac{1}{m}\sqrt{p_kq_k}(p_k-q_k) f(\, \cdot \, ,k) \Big)\\[35pt]
&= m^2(m-1)! \ellnorm{2}{m-1}{f( \, \cdot \, ,k)}^2\\
&\phantom{{}={}} +m^2 \sum_{\genfrac{}{}{0pt}{}{s=1}{s \neq m-1}}^{2(m-1)} J_s \Big( \sum_{r=1}^m \sum_{\ell=1}^r \1_{\{ 2m-r-\ell=s \}} (r-1)! \binom{m-1}{r-1}^2 \binom{r-1}{\ell-1}\\
&\phantom{\phantom{{}={}} +m^2 \sum_{\genfrac{}{}{0pt}{}{s=1}{s \neq m-1}}^{2(m-1)} J_s \Big( \sum_{r=1}^m \sum_{\ell=1}^r} \times (\widetilde{\varphi^{*r-\ell}(f(\, \cdot \, ,k) \star_{r-1}^{\ell-1} f(\, \cdot \, ,k))}) \1_{\Delta_{2m-r-\ell}} \Big)\\
&\phantom{{}={}} +m^2 J_{m-1} \Big( \sum_{r=1}^m \sum_{\ell=1}^r \1_{\{ 2m-r-\ell=m-1 \}} (r-1)! \binom{m-1}{r-1}^2 \binom{r-1}{\ell-1}\\
&\phantom{\phantom{{}={}} +m^2 J_{m-1} \Big( \sum_{r=1}^m \sum_{\ell=1}^r} \times (\widetilde{\varphi^{*r-\ell}(f(\, \cdot \, ,k) \star_{r-1}^{\ell-1} f(\, \cdot \, ,k))}) \1_{\Delta_{2m-r-\ell}}\\
&\phantom{\phantom{{}={}} +m^2 J_{m-1} \Big(} +\frac{1}{m}\sqrt{p_kq_k}(p_k-q_k) f(\, \cdot \, ,k) \Big).
\end{align*}
Thus, by the isometry formula in \eqref{Isometry formula},
\begin{align*}
&\E[((D_kF)^2 + \sqrt{p_kq_k}(p_k-q_k)D_kF)^2] \\[5pt]
&= m^4((m-1)! \ellnorm{2}{m-1}{f( \, \cdot \, ,k)}^2)^2\\
&\phantom{{}={}}+ m^4 \sum_{\genfrac{}{}{0pt}{}{s=1}{s \neq m-1}}^{2(m-1)} s! \Big\| \sum_{r=1}^m \sum_{\ell=1}^r \1_{\{ 2m-r-\ell=s \}} (r-1)! \binom{m-1}{r-1}^2 \binom{r-1}{\ell-1}\\
&\phantom{\phantom{{}={}}+ m^4 \sum_{\genfrac{}{}{0pt}{}{s=1}{s \neq m-1}}^{2(m-1)} s! \Big\| \sum_{r=1}^m \sum_{\ell=1}^r} \times (\widetilde{\varphi^{*r-\ell}(f(\, \cdot \, ,k) \star_{r-1}^{\ell-1} f(\, \cdot \, ,k))}) \1_{\Delta_{2m-r-\ell}} \Big\|_{\ell^2(\N)^{\otimes s}}^2\\
&\phantom{{}={}}+ m^4(m-1)! \Big\| \sum_{r=1}^m \sum_{\ell=1}^r \1_{\{ 2m-r-\ell=m-1 \}} (r-1)! \binom{m-1}{r-1}^2 \binom{r-1}{\ell-1}\\
&\phantom{\phantom{{}={}}+ m^4(m-1)! \Big\| \sum_{r=1}^m \sum_{\ell=1}^r} \times (\widetilde{\varphi^{*r-\ell}(f(\, \cdot \, ,k) \star_{r-1}^{\ell-1} f(\, \cdot \, ,k))}) \1_{\Delta_{2m-r-\ell}}\\
&\phantom{\phantom{{}={}}+ m^4(m-1)! \Big\|} +\frac{1}{m}\sqrt{p_kq_k}(p_k-q_k) f(\, \cdot \, ,k) \Big\|_{\ell^2(\N)^{\otimes m-1}}^2,
\end{align*}
and therefore,
\begin{align*}
A_3 &\leq (m!\ellnorm{2}{m}{f}^2)^{1/2}\\
&\phantom{{}\leq{}} \times \Big( m^3\sum_{k=1}^\infty \frac{1}{p_kq_k} ((m-1)! \ellnorm{2}{m-1}{f( \, \cdot \, ,k)}^2)^2\\
&\phantom{\phantom{{}\leq{}} \times \Big(} + m^3\sum_{k=1}^\infty \frac{1}{p_kq_k} \sum_{\genfrac{}{}{0pt}{}{s=1}{s \neq m-1}}^{2(m-1)} s! \Big\| \sum_{r=1}^m \sum_{\ell=1}^r \1_{\{ 2m-r-\ell=s \}} (r-1)! \binom{m-1}{r-1}^2 \binom{r-1}{\ell-1}\\ &\phantom{\phantom{\phantom{{}\leq{}} \times \Big(} + m^3\sum_{k=1}^\infty \frac{1}{p_kq_k} \sum_{\genfrac{}{}{0pt}{}{s=1}{s \neq m-1}}^{2(m-1)} s! \Big\| \sum_{r=1}^m \sum_{\ell=1}^r} \times (\widetilde{\varphi^{*r-\ell}(f(\, \cdot \, ,k) \star_{r-1}^{\ell-1} f(\, \cdot \, ,k))}) \1_{\Delta_{2m-r-\ell}} \Big\|_{\ell^2(\N)^{\otimes s}}^2\\
&\phantom{\phantom{{}\leq{}} \times \Big(} + m^3\sum_{k=1}^\infty \frac{1}{p_kq_k} (m-1)! \Big\| \sum_{r=1}^m \sum_{\ell=1}^r \1_{\{ 2m-r-\ell=m-1 \}} (r-1)! \binom{m-1}{r-1}^2 \binom{r-1}{\ell-1}\\
&\phantom{\phantom{\phantom{{}\leq{}} \times \Big(} + m^3\sum_{k=1}^\infty \frac{1}{p_kq_k} (m-1)! \Big\| \sum_{r=1}^m \sum_{\ell=1}^r} \times (\widetilde{\varphi^{*r-\ell}(f(\, \cdot \, ,k) \star_{r-1}^{\ell-1} f(\, \cdot \, ,k))}) \1_{\Delta_{2m-r-\ell}}\\
&\phantom{\phantom{\phantom{{}\leq{}} \times \Big(} + m^3\sum_{k=1}^\infty \frac{1}{p_kq_k} (m-1)! \Big\|} +\frac{1}{m}\sqrt{p_kq_k}(p_k-q_k) f(\, \cdot \, ,k) \Big\|_{\ell^2(\N)^{\otimes m-1}}^2 \Big)^{1/2}.
\end{align*}
The result now follows by a final application of the isometry formula in \eqref{Isometry formula} to deduce
\begin{align*}
\Var(F) = m!\norm{f}_{\ell^2(\N)^{\otimes m}}^2.
\end{align*}
\end{proof}

\begin{remark}\rm\label{J_2 remark}
Resorting, e.g., to the case $m=2$ in Theorem \ref{J_m theorem} yields the bound
\begin{align*}
&d_{TV}(F, \Po(\lambda))\\[5pt]
&\leq \Big( 1 \wedge \sqrt{\frac{2}{e\lambda}} \Big) \absolute{\lambda - \E[F]} + \frac{1-e^{-\lambda}}{\lambda} \absolute{\lambda - \Var(F)}\\
&\phantom{{}\leq{}} +\frac{1-e^{-\lambda}}{\lambda} (4 \| \varphi^{*1}(f \star_2^1 f) \|_{\ell^2(\N)}^2 + 8 \ellnorm{2}{2}{(f \star_1^1 f) \1_{\Delta_2}}^2)^{1/2}\\
&\phantom{{}\leq{}} +\frac{1-e^{-\lambda}}{\lambda} \sqrt{\Var(F)}\\
&\phantom{{}\leq{}+{}} \times \Big( 8\sum_{k=1}^\infty \frac{1}{p_kq_k} \| f( \, \cdot \, ,k) \|_{\ell^2(\N)}^4 + 16\sum_{k=1}^\infty \frac{1}{p_kq_k} \ellnorm{2}{2}{(f(\, \cdot \, ,k) \star_0^0 f(\, \cdot \, ,k)) \1_{\Delta_2}}^2\\
&\phantom{\phantom{{}\leq{}+{}} \times \Big(} + 8\sum_{k=1}^\infty \frac{1}{p_kq_k} \| \varphi^{*1}(f(\, \cdot \, ,k) \star_1^0 f(\, \cdot \, ,k)) + \frac{1}{2}\sqrt{p_kq_k}(p_k-q_k) f(\, \cdot \, ,k) \|_{\ell^2(\N)}^2 \Big)^{1/2}.
\end{align*}
Thus, the weak convergence of the law of $F_n = \E[F_n] + J_2(f_n)$ with $f_n \in \ell_0^2(\N)^{\circ 2}$, for every $n \in \N$, to a Poisson distribution is implied by the convergence of the first two moments of $F_n$ and by the vanishing of the quantities
\begin{align*}
\| \varphi^{*1}(f_n \star_2^1 f_n) \|_{\ell^2(\N)}^2, \quad \ellnorm{2}{2}{(f_n \star_1^1 f_n) \1_{\Delta_2}}^2,
\end{align*}
\begin{align*}
\sum_{k=1}^\infty \frac{1}{p_kq_k} \| f_n( \, \cdot \, ,k) \|_{\ell^2(\N)}^4, \quad \sum_{k=1}^\infty \frac{1}{p_kq_k} \ellnorm{2}{2}{(f_n(\, \cdot \, ,k) \star_0^0 f_n(\, \cdot \, ,k)) \1_{\Delta_2}}^2
\end{align*}
and
\begin{align*}
\sum_{k=1}^\infty \frac{1}{p_kq_k} \| \varphi^{*1}(f_n(\, \cdot \, ,k) \star_1^0 f_n(\, \cdot \, ,k)) + \frac{1}{2}\sqrt{p_kq_k}(p_k-q_k) f_n(\, \cdot \, ,k) \|_{\ell^2(\N)}^2,
\end{align*}
as $n \rightarrow \infty$.
\end{remark}

\begin{corollary}\label{J_2 corollary}
Let $n \geq 2$ be an integer, $p_k = 1 - q_k := \frac{1}{n}$, for every $k \in \N$, and $F_n := J_2(f_n)$ with $f_n \in \ell_0^2(\N)^{\circ 2}$ given by
\begin{align*}
f_n(i,j) :=
\begin{cases} \frac{n-1}{2n^2}, &\text{if $(i,j) \in \{(1,2), (2,1), \dotsc, (1,n), (n,1)\}$},\\
0, &\text{otherwise}.
\end{cases}
\end{align*}
Furthermore, let $\Po(\lambda_n)$ be a Poisson random variable with mean $\lambda_n := \Var(F_n)$. Then,
\begin{align*}
d_{TV}(F_n, \Po(\lambda_n)) \leq \frac{C}{\sqrt{n}}
\end{align*}
with $C := \frac{5}{2} + \sqrt{2}$.
\end{corollary}
\begin{proof}
First note that $F_n$ fulfills the assumptions of Theorem \ref{J_m theorem}. To see that $F_n$ only takes values in $\N_0$, let $(B_k)_{k \in \N}$ be a sequence of independent Bernoulli random variables with $P(B_k=1)=\frac{1}{n}$ and $P(B_k=0)=1-\frac{1}{n}$, for every $k \in \N$. Then, $Y_k \overset{d}{=} \frac{B_k-p_k}{\sqrt{p_kq_k}} = \frac{nB_k-1}{\sqrt{n-1}}$, for every $k \in \N$, and thus,
\begin{align}\label{Strictly positive integer}
F_n = \sum_{i,j=1}^n f_n(i,j)Y_iY_j = \frac{n-1}{n^2}Y_1\sum_{i=2}^n Y_i \overset{d}{=} (B_1-n)\sum_{i=2}^n (B_i-n).
\end{align}
Now, since $n \geq 2$, we have that $B_i - n$ is a strictly negative integer, for every $i = 1, \dotsc, n$. Therefore, it follows from \eqref{Strictly positive integer} that $F_n$ is a strictly positive integer. Let us come to the proof of the assertion. According to Remark \ref{J_2 remark}, we have to further compute the quantities
\begin{align*}
A_1(n) := \absolute{\lambda_n-\E[F_n]}, \quad A_2(n) := \absolute{\lambda_n-\Var(F_n)}, \quad A_3(n) := \| \varphi^{*1}(f_n \star_2^1 f_n) \|_{\ell^2(\N)}^2,
\end{align*}
\begin{align*}
A_4(n) := \ellnorm{2}{2}{(f_n \star_1^1 f_n) \1_{\Delta_2}}^2, \quad A_5(n) := \sum_{k=1}^\infty \frac{1}{p_kq_k} \| f_n( \, \cdot \, ,k) \|_{\ell^2(\N)}^4,
\end{align*}
\begin{align*}
A_6(n) := \sum_{k=1}^\infty \frac{1}{p_kq_k} \ellnorm{2}{2}{(f_n(\, \cdot \, ,k) \star_0^0 f_n(\, \cdot \, ,k)) \1_{\Delta_2}}^2
\end{align*}
and
\begin{align*}
A_7(n) := \sum_{k=1}^\infty \frac{1}{p_kq_k} \| \varphi^{*1}(f_n(\, \cdot \, ,k) \star_1^0 f_n(\, \cdot \, ,k)) + \frac{1}{2}\sqrt{p_kq_k}(p_k-q_k) f_n(\, \cdot \, ,k) \|_{\ell^2(\N)}^2.
\end{align*}
First of, since $\lambda_n = \Var(F_n) = 2\ellnorm{2}{2}{f_n}^2 = \frac{(n-1)^3}{n^4}$, we have that $A_1(n) = \frac{(n-1)^3}{n^4} \leq \frac{1}{n}$ and $A_2(n)=0$. Considering $A_3(n)$, for every $i \in \N$, we get
\begin{align*}
f_n \star_2^1 f_n (i) = \sum_{j=1}^n f_n^2(i,j) = f_n^2(i,1) + \sum_{j=2}^n f_n^2(i,j) = \frac{(n-1)^2}{4n^4}\Big( \1_{\{ i=2, \dotsc, n\}} + (n-1)\1_{\{ i=1\}} \Big),
\end{align*}
and hence, with $\varphi_k^2 = \frac{(q_k-p_k)^2}{p_kq_k} = \frac{(n-2)^2}{n-1}$, for every $k \in \N$,
\begin{align*}
A_3(n) &= \frac{(n-1)^3(n-2)^2}{16n^8} \sum_{i=1}^n \Big( \1_{\{ i=2, \dotsc, n\}} + (n-1)^2\1_{\{ i=1\}} \Big) = \frac{(n-1)^4(n-2)^2}{16n^7} \leq \frac{1}{16n}.
\end{align*}
Turning to $A_4(n)$, for every $i,j \in \N$, we have
\begin{align*}
&(f_n \star_1^1 f_n)\1_{\Delta_2}(i,j) = \sum_{k=1}^n f_n(i,k)f_n(j,k)\1_{\Delta_2}(i,j)\\
&= f_n(i,1)f_n(j,1)\1_{\Delta_2}(i,j) + \sum_{k=2}^n f_n(i,k)f_n(j,k)\1_{\Delta_2}(i,j) = \frac{(n-1)^2}{4n^4}\1_{\{ 2 \leq i \neq j \leq n\}},
\end{align*}
and thus,
\begin{align*}
A_4(n) = \frac{(n-1)^5(n-2)}{16n^8} \leq \frac{1}{16n^2}.
\end{align*}
To compute $A_5(n)$, note that, for every $k \in \N$,
\begin{align*}
\| f_n( \, \cdot \, ,k) \|_{\ell^2(\N)}^4 = \Big( f_n^2(1,k) + \sum_{j=2}^n f_n^2(j,k) \Big)^2 = \frac{(n-1)^4}{16n^8} \Big( \1_{\{ k=2, \dotsc, n \}} + (n-1)^2 \1_{\{ k=1 \}} \Big),
\end{align*}
and therefore, with $\frac{1}{p_kq_k} = \frac{n^2}{n-1}$, for every $k \in \N$,
\begin{align*}
A_5(n) = \frac{(n-1)^3}{16n^6} \sum_{k=1}^n \Big( \1_{\{ k=2, \dotsc, n \}} + (n-1)^2 \1_{\{ k=1 \}} \Big) = \frac{(n-1)^4}{16n^5} \leq \frac{1}{16n}.
\end{align*}
For $A_6(n)$, it shows that, for every $i,j,k \in \N$,
\begin{align*}
(f_n(\, \cdot \, ,k) \star_0^0 f_n(\, \cdot \, ,k)) \1_{\Delta_2}(i,j) = f_n(i,k) f_n(j,k) \1_{\Delta_2}(i,j) = \frac{(n-1)^2}{4n^4}\1_{\{ 2 \leq i \neq j \leq n\}} \1_{\{k=1\}}.
\end{align*}
Furthermore, for every $k \in \N$,
\begin{align*}
\ellnorm{2}{2}{(f_n(\, \cdot \, ,k) \star_0^0 f_n(\, \cdot \, ,k)) \1_{\Delta_2}}^2 = \frac{(n-1)^5(n-2)}{16n^8}\1_{\{k=1\}},
\end{align*}
so that, again with $\frac{1}{p_kq_k} = \frac{n^2}{n-1}$, for every $k \in \N$,
\begin{align*}
A_6(n) = \frac{(n-1)^4(n-2)}{16n^6} \leq \frac{1}{16n}.
\end{align*}
Finally, considering $A_7(n)$, for every $k \in \N$, it holds that
\begin{align*}
&\| \varphi^{*1}(f_n(\, \cdot \, ,k) \star_1^0 f_n(\, \cdot \, ,k)) + \frac{1}{2}\sqrt{p_kq_k}(p_k-q_k) f_n(\, \cdot \, ,k) \|_{\ell^2(\N)}^2\\
&= \sum_{j=1}^n \Big( \frac{n-2}{\sqrt{n-1}} f_n^2(j,k) - \frac{\sqrt{n-1}(n-2)}{2n^2} f_n(j,k) \Big)^2\\
&= \Big( \frac{(n-1)^{3/2}(n-2)}{4n^4} - \frac{(n-1)^{3/2}(n-2)}{4n^4} \Big)^2 \Big( \1_{\{ k=2, \dotsc, n \}} + (n-1) \1_{\{ k=1 \}} \Big) = 0,
\end{align*}
where we used that $\varphi_k = \frac{n-2}{\sqrt{n-1}}$ and $\sqrt{p_kq_k}(p_k-q_k) = - \frac{\sqrt{n-1}(n-2)}{n^2}$, for every $k \in \N$. We thus conclude that $A_7(n)=0$.
Now, it follows from Remark \ref{J_2 remark} that
\begin{align*}
d_{TV}(F_n, \Po(\lambda_n)) &\leq A_1(n) + 2\sqrt{A_3(n)} + 2\sqrt{2 A_4(n)} + 2\sqrt{2A_5(n)} + 4\sqrt{A_6(n)},
\end{align*}
where we used that, for every $n \geq 2$, $\Big( 1 \wedge \sqrt{\frac{2}{e\lambda_n}} \Big) = 1$, $\frac{1-e^{-\lambda_n}}{\lambda_n} \leq 1$ and $\sqrt{\Var(F_n)} \leq 1$. This yields the assertion.
\end{proof}

\begin{remark}\rm
Note here that the Rademacher functional in Corollary \ref{J_2 corollary} is of the same spirit as the one considered in the example that follows Proposition 8.3 in \cite{PriTor}. For a sequence of success probabilities $p=(p_k)_{k \in \N}$ as in Corollary \ref{J_2 corollary}, the authors of \cite{PriTor} compare a suitably shifted stochastic double integral $F_n = \lambda_n + J_2(f_n)$ to a Poisson random variable $\Po(\lambda_n)$, where $\lambda_n \geq 4n$ is an integer and $f_n \in \ell_0^2(\N)^{\circ 2}$ is given by
\begin{align*}
f_n(i,j) :=
\begin{cases} \frac{n-1}{n}, &\text{if $(i,j) \in \{(1,2), (2,1), \dotsc, (1,n), (n,1)\}$},\\
0, &\text{otherwise}.
\end{cases}
\end{align*}
However, for this particular choice our bound in Remark \ref{J_2 remark} as well as the corresponding bound in Proposition 8.3 in \cite{PriTor} does not vanish, as $n \rightarrow \infty$, since the involved norms of contractions do not tend to zero. For example, we have
\begin{align*}
2\ellnorm{2}{2}{f_n}^2 = \frac{4(n-1)^3}{n^2} \quad \text{and} \quad \ellnorm{2}{2}{f_n \star_1^1 f_n}^2 = \frac{2(n-1)^6}{n^4},
\end{align*}
so that the quantities $\frac{1-e^{-\lambda_n}}{\lambda_n} \absolute{\lambda_n - 2\ellnorm{2}{2}{f_n}^2}$ and $\frac{1-e^{-\lambda_n}}{\lambda_n} \ellnorm{2}{2}{f_n \star_1^1 f_n}^2$ in the bound of Proposition 8.3 in \cite{PriTor} never vanish at the same time, no matter of the choice of $\lambda_n$.
\end{remark}

We will now turn to our final result, a second order Poincar\'{e} type bound for the Poisson approximation of Rademacher functionals. One advantage of such a bound is that it can be further evaluated without the use of a product formula for multiple stochastic integrals or even a specification of the chaos representation of the Rademacher functional of interest as in \eqref{Chaos representation}. See, e.g., Theorem 1.1 in \cite{KroReiThae2} for an efficient application of a corresponding second order Poincar\'{e} type bound for the normal approximation of Rademacher functionals. Before we come to the statement, we collect some tools from \cite{KroReiThae2}.

\begin{lemma}[cf.\ Proposition 3.3 in \cite{KroReiThae2}]\label{Mehler inequality}
For $m \in \N$, let $k_1, \dotsc, k_m \in \N$ and $F \in \dom(D^m)$. Then, for every real $\alpha \geq 1$,
\begin{align*}
\E[\absolute{D_{k_1, \dotsc, k_m}^m L^{-1}(F-\E[F])}^\alpha] \leq \E[\absolute{D_{k_1, \dotsc, k_m}^m F}^\alpha].
\end{align*}
\end{lemma}

\begin{lemma}[cf.\ Proposition 3.4 and Remark 3.1 in \cite{KroReiThae2}]\label{Poincare inequality}
Let $F \in L^1(\Omega)$. Then,
\begin{align*}
\Var(F) \leq \E[\norm{DF}_{\ell^2(\N)}^2].
\end{align*}
\end{lemma}

\begin{theorem}\label{Second order Poincare inequality}
Let $F \in \dom(D^2)$ with values in $\N_0$ and let $\Po(\lambda)$ be a Poisson random variable with mean $\lambda > 0$. Then,
\begin{align}\label{Second order Poincare inequality equation}
&d_{TV}(F, \Po(\lambda)) \notag\\[5pt]
&\leq \Big( 1 \wedge \sqrt{\frac{2}{e\lambda}} \Big) \absolute{\lambda - \E[F]} + \frac{1-e^{-\lambda}}{\lambda} \absolute{\lambda - \Var(F)}\notag\\
&\phantom{{}\leq{}}+ \frac{1-e^{-\lambda}}{\lambda} \Big( \frac{15}{4} \sum_{j,k,\ell=1}^\infty (\E[(D_jF)^2(D_kF)^2])^{1/2}(\E[(D_\ell D_jF)^2(D_\ell D_kF)^2])^{1/2} \Big)^{1/2}\notag\\
&\phantom{{}\leq{}}+ \frac{1-e^{-\lambda}}{\lambda} \Big( \frac{3}{4} \sum_{j,k,\ell=1}^\infty \frac{1}{p_\ell q_\ell} \E[(D_\ell D_jF)^2(D_\ell D_kF)^2] \Big)^{1/2}\notag\\
&\phantom{{}\leq{}}+ \frac{1-e^{-\lambda}}{\lambda} \sum_{k=1}^\infty \frac{1}{\sqrt{p_kq_k}} (\E[(D_kF)^2 (D_kF + \sqrt{p_kq_k}(p_k-q_k))^2])^{1/2} (\E[(D_kF)^2])^{1/2}.
\end{align}
\end{theorem}
\begin{proof}
We build on Corollary \ref{Main corollary} by further estimating the quantities
\begin{align*}
A_1 := \E[\absolute{\lambda - \langle DF, -DL^{-1}(F-\E[F]) \rangle_{\ell^2(\N)}}]
\end{align*}
and
\begin{align*}
A_2 := \sum_{k=1}^\infty \frac{1}{\sqrt{p_kq_k}} \E[D_kF (D_kF + \sqrt{p_kq_k}(p_k-q_k)) \cdot \absolute{-D_kL^{-1}(F-\E[F])}].
\end{align*}
Starting with $A_1$, by means of the triangle and the Cauchy-Schwarz inequality, we get
\begin{align}\label{Second order Poincare inequality proof equation 1}
A_1 &\leq \E[\absolute{\lambda - \Var(F)}] + \E[\absolute{\Var(F) - \langle DF, DL^{-1}(F-\E[F]) \rangle_{\ell^2(\N)}}]\notag\\
&\leq \E[\absolute{\lambda - \Var(F)}] + (\E[(\Var(F) - \langle DF, DL^{-1}(F-\E[F]) \rangle_{\ell^2(\N)})^2])^{1/2}.
\end{align}
Note that, by choosing $G=F-\E[F]$ in the integration by parts formula in \eqref{Integration by parts formula}, we have
\begin{align*}
\Var(F) = \E[\langle DF, DL^{-1}(F-\E[F]) \rangle_{\ell^2(\N)})^2],
\end{align*}
and thus,
\begin{align*}
\E[(\Var(F) - \langle DF, DL^{-1}(F-\E[F]) \rangle_{\ell^2(\N)})^2] = \Var(\langle DF, DL^{-1}(F-\E[F]) \rangle_{\ell^2(\N)}).
\end{align*}
Hence, the second summand on the right hand side of \eqref{Second order Poincare inequality proof equation 1} can be further estimated by Lemma \ref{Poincare inequality} and Lemma \ref{Mehler inequality} as shown in the proof of Theorem 4.1 in \cite{KroReiThae2}, which leads to
\begin{align*}
A_1 &\leq \E[\absolute{\lambda - \Var(F)}] + \Big( \frac{15}{4} \sum_{j,k,\ell=1}^\infty (\E[(D_jF)^2(D_kF)^2])^{1/2}(\E[(D_\ell D_jF)^2(D_\ell D_kF)^2])^{1/2} \Big)^{1/2}\\
&\phantom{{}\leq{}}+ \Big( \frac{3}{4} \sum_{j,k,\ell=1}^\infty \frac{1}{p_\ell q_\ell} \E[(D_\ell D_jF)^2(D_\ell D_kF)^2] \Big)^{1/2}.
\end{align*}
Furthermore, by virtue of the Cauchy-Schwarz inequality and Lemma \ref{Mehler inequality}, we get
\begin{align*}
A_2 &\leq \sum_{k=1}^\infty \frac{1}{\sqrt{p_kq_k}} (\E[(D_kF)^2 (D_kF + \sqrt{p_kq_k}(p_k-q_k))^2])^{1/2} (\E[(D_kL^{-1}(F-\E[F]))^2])^{1/2}\\
&\leq \sum_{k=1}^\infty \frac{1}{\sqrt{p_kq_k}} (\E[(D_kF)^2 (D_kF + \sqrt{p_kq_k}(p_k-q_k))^2])^{1/2} (\E[(D_kF)^2])^{1/2}.
\end{align*}
This concludes the proof.
\end{proof}

\begin{remark}\rm
To give a first application and an insight into the quality of the bound in Theorem \ref{Second order Poincare inequality}, we consider the Poisson approximation of infinite sums of Bernoulli random variables once more. For this, let $(B_k)_{k \in \N}$ be a sequence of independent Bernoulli random variables with $P(B_k=1)=p_k$ and $P(B_k=0)=q_k$, for every $k \in \N$, and $\sum_{k=1}^\infty p_k < \infty$, and let $F:=\sum_{k=1}^\infty B_k$. Recall from the proof of Corollary \ref{Bernoulli sums corollary} that
\begin{align*}
F \stackrel{d}{=} \sum_{k=1}^\infty \frac{X_k+1}{2}.
\end{align*}
Now, for every $k \in \N$, we have that
\begin{align*}
F_k^+ \stackrel{d}{=} 1 + \sum_{\substack{\ell = 1\\ \ell \neq k}}^\infty \frac{X_\ell+1}{2} \quad \text{and} \quad F_k^- \stackrel{d}{=} \sum_{\substack{\ell = 1\\ \ell \neq k}}^\infty \frac{X_\ell+1}{2},
\end{align*}
and therefore, for every $k,\ell \in \N$, we get
\begin{align*}
D_kF = \sqrt{p_kq_k}(F_k^+ - F_k^-) = \sqrt{p_kq_k} \quad \text{and} \quad D_\ell D_kF = 0,
\end{align*}
$P$-almost surely.
Hence, $F \in \dom(D^2)$ by \eqref{F in dom(D)} and all assumptions of Theorem \ref{Second order Poincare inequality} are fulfilled. According to this, we have to further compute the quantities
\begin{align*}
A_1 := \absolute{\lambda - \E[F]}, \quad A_2 := \absolute{\lambda - \Var(F)},
\end{align*}
\begin{align*}
A_3 := \Big( \sum_{j,k,\ell=1}^\infty (\E[(D_jF)^2(D_kF)^2])^{1/2}(\E[(D_\ell D_jF)^2(D_\ell D_kF)^2])^{1/2} \Big)^{1/2},
\end{align*}
\begin{align*}
A_4 := \Big( \sum_{j,k,\ell=1}^\infty \frac{1}{p_\ell q_\ell} \E[(D_\ell D_jF)^2(D_\ell D_kF)^2] \Big)^{1/2}
\end{align*}
and
\begin{align*}
A_5 := \sum_{k=1}^\infty \frac{1}{\sqrt{p_kq_k}} (\E[(D_kF)^2 (D_kF + \sqrt{p_kq_k}(p_k-q_k))^2])^{1/2} (\E[(D_kF)^2])^{1/2}
\end{align*}
\end{remark}
from the bound in \eqref{Second order Poincare inequality equation}. Recall from the proof of Corollary \ref{Bernoulli sums corollary} that
\begin{align*}
A_1 = \Bigabsolute{\lambda - \sum_{k=1}^\infty p_k} \quad \text{and} \quad A_2 = \Bigabsolute{\lambda - \sum_{k=1}^\infty p_kq_k}.
\end{align*}
Furthermore, $A_3=A_4=0$ and
\begin{align*}
A_5 &= \sum_{k=1}^\infty p_kq_k (1 + p_k-q_k) = 2\sum_{k=1}^\infty p_k^2q_k.
\end{align*}
This leads to the exact same bound
\begin{align*}
d_{TV}(F, \Po(\lambda)) \leq \Big( 1 \wedge \sqrt{\frac{2}{e\lambda}} \Big) \Bigabsolute{\lambda - \sum_{k=1}^\infty p_k} + \frac{1-e^{-\lambda}}{\lambda} \Bigabsolute{\lambda - \sum_{k=1}^\infty p_kq_k} + \frac{2(1-e^{-\lambda})}{\lambda} \sum_{k=1}^\infty p_k^2q_k
\end{align*}
that we have deduced directly from Theorem \ref{Main theorem} in Corollary \ref{Bernoulli sums corollary}.

\section{Appendix}
The purpose of this appendix is to prove Proposition \ref{Multiplication formula proposition}. We start by collecting some arguments that will be used within the proof. Note that Lemma \ref{Contraction limit} is a slight generalization of Lemma 2.6 in \cite{NouPecRei}, while Lemma \ref{Martingale lemma} is known as Lemma 4.6 in \cite{Pri2}.

\begin{lemma}\label{Contraction limit}
Fix $n,m \in \N$. Furthermore, let $(f_k)_{k \in \N}$ and $(g_k)_{k \in \N}$ be two sequences of kernels with $f_k \in \ell_0^2(\N)^{\circ n}$ and $g_k \in \ell_0^2(\N)^{\circ m}$, for every $k \in \N$. Then, if  $(f_k)_{k \in \N}$ converges to a kernel $f$ in $\ell_0^2(\N)^{\circ n}$ and $(g_k)_{k \in \N}$ converges to a kernel $g$ in $\ell_0^2(\N)^{\circ m}$, it holds that, for every $r=0, \dotsc, n \wedge m$ and $\ell = 0, \dotsc, r$, the sequence of contractions $(f_k \star_r^{\ell} g_k)_{k \in \N}$ converges to $f \star_r^{\ell} g$ in $\ell^2(\N)^{\otimes n+m-r-\ell}$.
\end{lemma}
\begin{proof} Using the triangle inequality as well as Lemma 2.4 in \cite{NouPecRei}, we see that
\begin{align*}
&\ellnorm{2}{n+m-r-\ell}{f_k \star_r^{\ell} g_k - f \star_r^{\ell} g}\\
&= \ellnorm{2}{n+m-r-\ell}{f_k \star_r^{\ell} (g_k-g+g) - (f-f_k+f_k) \star_r^{\ell} g}\\
&= \ellnorm{2}{n+m-r-\ell}{f_k \star_r^{\ell} (g_k-g) + (f_k-f) \star_r^{\ell} g}\\
&\leq \ellnorm{2}{m+n-r-\ell}{f_k \star_r^{\ell} (g_k-g)} + \ellnorm{2}{m+n-r-\ell}{(f_k-f) \star_r^{\ell} g}\\
&\leq \ellnorm{2}{n}{f_k} \ellnorm{2}{m}{g_k-g} + \ellnorm{2}{n}{f_k-f} \ellnorm{2}{m}{g}.
\end{align*}
The statement now follows immediately by taking the limit $k \rightarrow \infty$.
\end{proof}

\begin{lemma}\label{Martingale lemma}
Let $n \in \N$ and $f \in \ell_0^2(\N)^{\circ n}$. Consider the sequence of truncated kernels $(f_k)_{k \in \N}$ with $f_k:=f\1_{\{ 1, \dotsc, k\}^n}$, for every $k \in \N$. Then, for every $k \in \N$,
\begin{align*}
J_n(f_k)=\E[J_n(f) | \mathcal{F}_k],
\end{align*}
where $(\mathcal{F}_k)_{k \in \N}$ denotes the canonical filtration given by $\mathcal{F}_k:=\sigma(X_1, \dotsc, X_k)$, for every $k \in \N$.
\end{lemma}

\begin{corollary}\label{Approximation argument}
Let $n \in \N$ and $f \in \ell_0^2(\N)^{\circ n}$. Consider the sequence of truncated kernels $(f_k)_{k \in \N}$ with $f_k:=f\1_{\{ 1, \dotsc, k\}^n}$, for every $k \in \N$. Then, the sequence $(J_n(f_k))_{k \in \N}$ convergences to $J_n(f)$ in $L^2(\Omega)$.
\end{corollary}
\begin{proof}
By virtue of Lemma \ref{Martingale lemma}, $(J_n(f_k))_{k \in \N}$ is a martingale with respect to $(\mathcal{F}_k)_{k \in \N}$. Thus, the convergence of $(J_n(f_k))_{k \in \N}$ to $J_n(f)$ in $L^2(\Omega)$ immediately follows by the martingale convergence theorem.
\end{proof}

\begin{proof}[Proof of Proposition \ref{Multiplication formula proposition}]
Fix $d \in \N$. We start by proving \eqref{Multiplication formula equation 1} for stochastic integrals of kernels $f \in \ell_0^2(\N)^{\circ n}$ and $g \in \ell_0^2(\N)^{\circ m}$ with finite supports $\supp(f) \subseteq \{1, \dotsc, d\}^n$ and $\supp(g) \subseteq \{1, \dotsc, d\}^m$. We put $\Delta_n^d:=\Delta_n \cap \{1, \dotsc, d\}^n$ and deduce from \eqref{Stochastic integral} that
\begin{align}\label{Multiplication formula proof equation 1}
J_n(f)J_m(g) = \sum_{(i_1, \dotsc, i_n, j_1, \dotsc, j_m) \in \Delta_n^d \times \Delta_m^d} f(i_1, \dotsc, i_n)g(j_1, \dotsc, j_m)Y_{i_1} \cdot \dotsc \cdot Y_{i_n}Y_{j_1} \cdot \dotsc \cdot Y_{j_m}.
\end{align}
We will now count the pairs of equal random variables in the products $Y_{i_1} \cdot \dotsc \cdot Y_{i_n}Y_{j_1} \cdot \dotsc \cdot Y_{j_m}$ in \eqref{Multiplication formula proof equation 1}. Since $(i_1, \dotsc, i_n) \in \Delta_n^d$ and $(j_1, \dotsc, j_m) \in \Delta_m^d$, each possible pair can only consist of one random variable taken from the set $\lbrace Y_{i_1}, \dotsc, Y_{i_n} \rbrace$ and one random variable taken from the set $\lbrace Y_{j_1}, \dotsc, Y_{j_m} \rbrace$. Thus, each product $Y_{i_1} \cdot \dotsc \cdot Y_{i_n}Y_{j_1} \cdot \dotsc \cdot Y_{j_m}$ can contain $r=0, \dotsc, n \wedge m$ pairs. Now, there are $r! \binom{n}{r} \binom{m}{r}$ different ways to build $r$ pairs as described above. (There are $\binom{n}{r}$ different ways to pick $r$ random variables from $\lbrace Y_{i_1}, \dotsc, Y_{i_n} \rbrace$, $\binom{m}{r}$ different ways to pick $r$ random variables from $\lbrace Y_{j_1}, \dotsc, Y_{j_m} \rbrace$ and finally $r!$ different ways to  group pairs from the two developed $r$-sets.) By the symmetry of the summands $f(i_1, \dotsc, i_n)g(j_1, \dotsc, j_m)Y_{i_1} \cdot \dotsc \cdot Y_{i_n}Y_{j_1} \cdot \dotsc \cdot Y_{j_m}$ in $i_1, \dotsc, i_n$ and $j_1, \dotsc, j_m$, respectively, the sum in \eqref{Multiplication formula proof equation 1} can be rewritten in terms of summands containing $r$ pairs of random variables
\begin{align}\label{Multiplication formula proof equation 2}
&J_n(f)J_m(g) \notag\\
&= \sum_{r=0}^{n \wedge m} r! \binom{n}{r} \binom{m}{r} \sum_{(\bm{i}_{n-r}, \bm{j}_{m-r}, \bm{k}_r) \in \Delta_{n+m-r}^d} f(\bm{i}_{n-r}, \bm{k}_r)g(\bm{j}_{m-r}, \bm{k}_r) \notag\\
&\phantom{{}= \sum_{r=0}^{n \wedge m} r! \binom{n}{r} \binom{m}{r} \sum_{(\bm{i}_{n-r}, \bm{j}_{m-r}, \bm{k}_r) \in \Delta_{n+m-r}^d}{}} \times Y_{i_1} \cdot \dotsc \cdot Y_{i_{n-r}}Y_{j_1} \cdot \dotsc \cdot Y_{j_{m-r}}Y_{k_1}^2 \cdot \dotsc \cdot Y_{k_r}^2
\end{align}
with $\bm{i}_{n-r}:=(i_1, \dotsc, i_{n-r})$, $\bm{j}_{m-r}:=(j_1, \dotsc, j_{m-r})$ and $\bm{k}_r:=(k_1, \dotsc, k_r)$. We will now further compute the product $Y_{k_1}^2 \cdot \dotsc \cdot Y_{k_r}^2$ in \eqref{Multiplication formula proof equation 2}. By \eqref{Structure equation} it follows that
\begin{align*}
\prod_{\ell=1}^r Y_{k_\ell}^2 &= \prod_{\ell=1}^r \left( 1+\varphi_{k_{\ell}}Y_{k_\ell} \right) = 1+\sum_{s=1}^r \sum_{1 \leq \ell_1 < \dotsc < \ell_s \leq r} \varphi_{k_{\ell_1}} \cdot \dotsc \cdot \varphi_{k_{\ell_s}}Y_{k_{\ell_1}} \cdot \dotsc \cdot Y_{k_{\ell_s}}.
\end{align*}
Thus, the inner sum in \eqref{Multiplication formula proof equation 2} can be rewritten as the sum of the two quantities
\begin{align}\label{Multiplication formula proof inner sum 1}
&\sum_{(\bm{i}_{n-r}, \bm{j}_{m-r}, \bm{k}_r) \in \Delta_{n+m-r}^d} f(\bm{i}_{n-r}, \bm{k}_r)g(\bm{j}_{m-r}, \bm{k}_r)Y_{i_1} \cdot \dotsc \cdot Y_{i_{n-r}}Y_{j_1} \cdot \dotsc \cdot Y_{j_{m-r}}
\end{align}
and
\begin{align}\label{Multiplication formula proof inner sum 2}
&\sum_{(\bm{i}_{n-r}, \bm{j}_{m-r}, \bm{k}_r) \in \Delta_{n+m-r}^d} \sum_{s=1}^r \sum_{1 \leq \ell_1 < \dotsc < \ell_s \leq r} \varphi_{k_{\ell_1}} \cdot \dotsc \cdot \varphi_{k_{\ell_s}}f(\bm{i}_{n-r}, \bm{k}_r)g(\bm{j}_{m-r}, \bm{k}_r) \notag\\
&\phantom{{}\sum_{(\bm{i}_{n-r}, \bm{j}_{m-r}, \bm{k}_r) \in \Delta_{n+m-r}^d} \sum_{s=1}^r \sum_{1 \leq \ell_1 < \dotsc < \ell_s \leq r}{}} \times Y_{i_1} \cdot \dotsc \cdot Y_{i_{n-r}}Y_{j_1} \cdot \dotsc \cdot Y_{j_{m-r}}Y_{k_{\ell_1}} \cdot \dotsc \cdot Y_{k_{\ell_s}}.
\end{align}
Using the fact that $f$ and $g$ vanish on diagonals as well as the symmetry of the product measure $\mu_{(Y,d)}^{\otimes n+m-2r}$ defined by $\mu_{(Y,d)}^{\otimes n+m-2r}(A):=\sum_{(i_1, \dotsc, i_{n+m-2r}) \in A}Y_{i_1} \cdot \dotsc \cdot Y_{i_{n+m-2r}}$, for every $A \in \{ 1, \dotsc, d \}^{n+m-2r}$, \eqref{Multiplication formula proof inner sum 1} can be further deduced as
\begin{align}\label{Multiplication formula proof inner sum 1 final}
&\sum_{(\bm{i}_{n-r}, \bm{j}_{m-r}, \bm{k}_r) \in \Delta_{n+m-r}^d} f(\bm{i}_{n-r}, \bm{k}_r)g(\bm{j}_{m-r}, \bm{k}_r)Y_{i_1} \cdot \dotsc \cdot Y_{i_{n-r}}Y_{j_1} \cdot \dotsc \cdot Y_{j_{m-r}} \notag\\
&=\sum_{(\bm{i}_{n-r}, \bm{j}_{m-r}) \in \Delta_{n+m-2r}^d} \sum_{\bm{k}_r \in \Delta_r^d} f(\bm{i}_{n-r}, \bm{k}_r)g(\bm{j}_{m-r}, \bm{k}_r)Y_{i_1} \cdot \dotsc \cdot Y_{i_{n-r}}Y_{j_1} \cdot \dotsc \cdot Y_{j_{m-r}} \notag\\
&=\sum_{(\bm{i}_{n-r}, \bm{j}_{m-r}) \in \Delta_{n+m-2r}^d} f \star_r^r g(\bm{i}_{n-r}, \bm{j}_{m-r})Y_{i_1} \cdot \dotsc \cdot Y_{i_{n-r}}Y_{j_1} \cdot \dotsc \cdot Y_{j_{m-r}} \notag\\
&=\sum_{(\bm{i}_{n-r}, \bm{j}_{m-r}) \in \Delta_{n+m-2r}^d} \widetilde{\left( f \star_r^r g \right)}(\bm{i}_{n-r}, \bm{j}_{m-r})Y_{i_1} \cdot \dotsc \cdot Y_{i_{n-r}}Y_{j_1} \cdot \dotsc \cdot Y_{j_{m-r}} \notag\\
&=J_{n+m-2r}\left( \widetilde{ \left( f \star_r^r g \right)} \1_{\Delta_{n+m-2r}^d} \right).
\end{align}
To further compute \eqref{Multiplication formula proof inner sum 2}, note that, due to the symmetry of $f$ and $g$, the summands
\begin{align*}
\varphi_{k_{\ell_1}} \cdot \dotsc \cdot \varphi_{k_{\ell_s}}f(\bm{i}_{n-r}, \bm{k}_r)g(\bm{j}_{m-r}, \bm{k}_r)Y_{i_1} \cdot \dotsc \cdot Y_{i_{n-r}}Y_{j_1} \cdot \dotsc \cdot Y_{j_{m-r}}Y_{k_{\ell_1}} \cdot \dotsc \cdot Y_{k_{\ell_s}}
\end{align*}
are symmetric in $k_1, \dotsc, k_r$. Thus, we get that, for $r=1, \dotsc, n \wedge m$,
\begin{align*}
&\sum_{\bm{k}_r \in \Delta_{r}^d} \sum_{1 \leq \ell_1 < \dotsc < \ell_s \leq r} \varphi_{k_{\ell_1}} \cdot \dotsc \cdot \varphi_{k_{\ell_s}}f(\bm{i}_{n-r}, \bm{k}_r)g(\bm{j}_{m-r}, \bm{k}_r)\\
&\phantom{{}\sum_{\bm{k}_r \in \Delta_{r}^d} \sum_{1 \leq \ell_1 < \dotsc < \ell_s \leq r}{}} \times Y_{i_1} \cdot \dotsc \cdot Y_{i_{n-r}}Y_{j_1} \cdot \dotsc \cdot Y_{j_{m-r}}Y_{k_{\ell_1}} \cdot \dotsc \cdot Y_{k_{\ell_s}}\\
&=\binom{r}{s} \sum_{\bm{k}_s \in \Delta_{s}^d} \varphi_{k_1} \cdot \dotsc \cdot \varphi_{k_s} f\star_r^{r-s}g(\bm{i}_{n-r}, \bm{k}_s, \bm{j}_{m-r})\\
&\phantom{{}=\binom{r}{s} \sum_{\bm{k}_s \in \Delta_{s}^d}{}} \times Y_{i_1} \cdot \dotsc \cdot Y_{i_{n-r}}Y_{j_1} \cdot \dotsc \cdot Y_{j_{m-r}}Y_{k_1} \cdot \dotsc \cdot Y_{k_s}\\
&=\binom{r}{s} \sum_{\bm{k}_s \in \Delta_{s}^d} \varphi^{*s} (f\star_r^{r-s}g) (\bm{i}_{n-r}, \bm{k}_s, \bm{j}_{m-r})\\
&\phantom{{}=\binom{r}{s} \sum_{\bm{k}_s \in \Delta_{s}^d}{}} \times Y_{i_1} \cdot \dotsc \cdot Y_{i_{n-r}}Y_{j_1} \cdot \dotsc \cdot Y_{j_{m-r}}Y_{k_1} \cdot \dotsc \cdot Y_{k_s}.
\end{align*}
Therefore, by using the same arguments as in \eqref{Multiplication formula proof inner sum 1 final}, we obtain for \eqref{Multiplication formula proof inner sum 2} that, if $r=1, \dotsc n \wedge m$,
\begin{align}\label{Multiplication formula proof inner sum 2 final 1}
&\sum_{(\bm{i}_{n-r}, \bm{j}_{m-r}, \bm{k}_r) \in \Delta_{n+m-r}^d} \sum_{s=1}^r \sum_{1 \leq \ell_1 < \dotsc < \ell_s \leq r} \varphi_{k_{\ell_1}} \cdot \dotsc \cdot \varphi_{k_{\ell_s}}f(\bm{i}_{n-r}, \bm{k}_r)g(\bm{j}_{m-r}, \bm{k}_r) \notag\\
&\phantom{{}\sum_{(\bm{i}_{n-r}, \bm{j}_{m-r}, \bm{k}_r) \in \Delta_{n+m-r}^d} \sum_{s=1}^r \sum_{1 \leq \ell_1 < \dotsc < \ell_s \leq r}{}} \times Y_{i_1} \cdot \dotsc \cdot Y_{i_{n-r}}Y_{j_1} \cdot \dotsc \cdot Y_{j_{m-r}}Y_{k_{\ell_1}} \cdot \dotsc \cdot Y_{k_{\ell_s}} \notag\\
&=\sum_{s=1}^r \binom{r}{s} \sum_{(\bm{i}_{n-r}, \bm{j}_{m-r}, \bm{k}_s) \in \Delta_{n+m-2r+s}^d} \varphi^{*s}(f\star_r^{r-s}g) (\bm{i}_{n-r}, \bm{k}_s, \bm{j}_{m-r}) \notag\\
&\phantom{{}=\sum_{s=1}^r \binom{r}{s} \sum_{(\bm{i}_{n-r}, \bm{j}_{m-r}, \bm{k}_s) \in \Delta_{n+m-2r+s}^d}{}} \times Y_{i_1} \cdot \dotsc \cdot Y_{i_{n-r}}Y_{j_1} \cdot \dotsc \cdot Y_{j_{m-r}}Y_{k_1} \cdot \dotsc \cdot Y_{k_s} \notag\\
&=\sum_{s=1}^r \binom{r}{s} \sum_{(\bm{i}_{n-r}, \bm{j}_{m-r}, \bm{k}_s) \in \Delta_{n+m-2r+s}^d}  \widetilde{(\varphi^{*s}(f\star_r^{r-s}g))} (\bm{i}_{n-r}, \bm{k}_s, \bm{j}_{m-r}) \notag\\
&\phantom{{}=\sum_{s=1}^r \binom{r}{s} \sum_{(\bm{i}_{n-r}, \bm{j}_{m-r}, \bm{k}_s) \in \Delta_{n+m-2r+s}^d}{}} \times Y_{i_1} \cdot \dotsc \cdot Y_{i_{n-r}}Y_{j_1} \cdot \dotsc \cdot Y_{j_{m-r}}Y_{k_1} \cdot \dotsc \cdot Y_{k_s} \notag\\
&=\sum_{s=1}^r \binom{r}{s} J_{n+m-2r+s} \Big( \widetilde{(\varphi^{*s}(f\star_r^{r-s}g))}\1_{\Delta_{n+m-2r+s}^d} \Big),
\end{align}
and, if $r=0$,
\begin{align}\label{Multiplication formula proof inner sum 2 final 2}
\sum_{(\bm{i}_{n-r}, \bm{j}_{m-r}, \bm{k}_r) \in \Delta_{n+m-r}^d} \sum_{s=1}^r \sum_{1 \leq \ell_1 < \dotsc < \ell_s \leq r} \varphi_{k_{\ell_1}} \cdot \dotsc \cdot \varphi_{k_{\ell_s}}f(\bm{i}_{n-r}, \bm{k}_r)g(\bm{j}_{m-r}, \bm{k}_r)=0.
\end{align}
Plugging \eqref{Multiplication formula proof inner sum 1 final}, \eqref{Multiplication formula proof inner sum 2 final 1} and \eqref{Multiplication formula proof inner sum 2 final 2} into \eqref{Multiplication formula proof equation 2} finally yields
\begin{align*}
J_n(f)J_m(g) &= \sum_{r=0}^{n \wedge m} r! \binom{n}{r} \binom{m}{r} J_{n+m-2r}\left( \widetilde{ \left( f \star_r^r g \right)} \1_{\Delta_{n+m-2r}^d} \right)\\
&\phantom{{}={}}+\sum_{r=1}^{n \wedge m} r! \binom{n}{r} \binom{m}{r} \sum_{s=1}^r \binom{r}{s} J_{n+m-2r+s} \Big( \widetilde{(\varphi^{*s}(f\star_r^{r-s}g))}\1_{\Delta_{n+m-2r+s}^d} \Big)\\[5pt]
&= \sum_{r=0}^{n \wedge m} r! \binom{n}{r} \binom{m}{r} J_{n+m-2r}\left( \widetilde{ \left( f \star_r^r g \right)} \1_{\Delta_{n+m-2r}^d} \right)\\
&\phantom{{}={}}+\sum_{r=1}^{n \wedge m} r! \binom{n}{r} \binom{m}{r} \sum_{\ell=0}^{r-1} \binom{r}{\ell} J_{n+m-r-\ell} \Big( \widetilde{(\varphi^{*r-\ell}(f\star_r^\ell g))}\1_{\Delta_{n+m-r-\ell}^d} \Big)\\[5pt]
&= \sum_{r=0}^{n \wedge m} r! \binom{n}{r} \binom{m}{r} \sum_{\ell=0}^r \binom{r}{\ell} J_{n+m-r-\ell} \Big( \widetilde{(\varphi^{*r-\ell}(f\star_r^\ell g))}\1_{\Delta_{n+m-r-\ell}^d} \Big)
\end{align*}
for stochastic integrals of kernels $f$ and $g$ with finite supports $\supp(f) \subseteq \{1, \dotsc, d\}^n$ and and $\supp(g) \subseteq \{1, \dotsc, d\}^m$. For the general case consider the sequences of truncated kernels $(f_d)_{d \in \N}$ and $(g_d)_{d \in \N}$ with $f_d:=f\1_{\{ 1, \dotsc, d\}^n}$ and $g_d:=g\1_{\{ 1, \dotsc, d\}^m}$, for every $d \in \N$. Now, $f_d \in \ell_0^2(\N)^{\circ n}$ with $\supp(f_d) \subseteq \{1, \dotsc, d\}^n$ and $g_d \in \ell_0^2(\N)^{\circ m}$ with $\supp(g_d) \subseteq \{1, \dotsc, d\}^m$, for every $d \in \N$. According to Lemma \ref{Contraction limit} and Corollary \ref{Approximation argument}, the statement now follows from the discussion above by taking the limit $d \rightarrow \infty$.
\end{proof}

\subsection*{Acknowledgement}
I would like to thank Peter Eichelsbacher for insightful discussions during the work on this paper.\newline
The author was supported by the German Research Foundation DFG via SFB-TR 12.

\bibliography{Poisson}
\end{document}